\documentclass[11pt,a4paper,reqno, draft]{amsart}
\usepackage{amsmath, amssymb, enumerate,  graphicx,tikz, stmaryrd, eufrak,enumitem,MnSymbol}

\usepackage{thmtools}

\usepackage[pdftex]{hyperref}
\usepackage{cleveref}
\usepackage{bbm}
\usepackage{cases}
\usepackage{array}
\usepackage{tikz-cd}
\usepackage[all]{xy}

\usepackage[hmarginratio={1:1},vmarginratio={1:1},lmargin=80.0pt,tmargin=90.0pt]{geometry}

\usepackage{tikz}
\tikzset{anchorbase/.style={baseline={([yshift=-0.5ex]current bounding box.center)}}}
\usetikzlibrary{decorations.markings}
\usetikzlibrary{decorations.pathreplacing}
\usetikzlibrary{arrows,shapes,positioning,backgrounds}
\tikzstyle directed=[postaction={decorate,decoration={markings,
    mark=at position #1 with {\arrow{>}}}}]
\tikzstyle rdirected=[postaction={decorate,decoration={markings,
    mark=at position #1 with {\arrow{<}}}}]
    
\usepackage{graphicx}
\usepackage[vcentermath]{youngtab}
\usepackage{nicefrac}

\numberwithin{equation}{section}

\newtheorem{theorem}[subsubsection]{Theorem}
\newtheorem{lemma}[theorem]{Lemma}
\newtheorem{prop}[theorem]{Proposition}
\newtheorem{corollary}[subsubsection]{Corollary}

\theoremstyle{definition}
\newtheorem{definition}[subsubsection]{Definition}

\newtheorem{remark}[theorem]{Remark}

\newtheorem{example}[subsubsection]{Example}

\newcommand{\fm}{\mathfrak{m}}
\newcommand{\p}{\mathfrak{p}}
\newcommand{\fp}{\mathfrak{p}}

\newcommand{\bA}{\mathsf{A}}

\newcommand{\bC}{\mathsf{C}}

\newcommand{\mO}{\mathbb{O}}

\newcommand{\cF}{\mathcal{F}}

\newcommand{\cL}{\mathcal{L}}

\newcommand{\red}{\mathrm{red}}

\newcommand{\Aff}{\mathsf{Aff}}

\newcommand{\Fais}{\mathsf{Fais}}
\newcommand{\Sch}{\mathsf{Sch}}
\newcommand{\Fun}{\mathsf{Fun}}
\newcommand{\Alg}{\mathsf{Alg}}
\newcommand{\LocAlg}{\mathsf{LocAlg}}
\newcommand{\Set}{\mathsf{Set}}
\newcommand{\Grp}{\mathsf{Grp}}
\newcommand{\Ab}{\mathsf{Ab}}

\newcommand{\Mod}{\mathsf{Mod}}
\newcommand{\Rep}{\mathsf{Rep}}

\newcommand{\QCoh}{\mathsf{QCoh}}
\newcommand{\LRS}{\mathsf{Lrs}}

\newcommand{\tto}{\twoheadrightarrow}
\newcommand{\cO}{\mathcal{O}}
\newcommand{\cJ}{\mathcal{J}}

\newcommand{\mN}{\mathbb{N}}
\newcommand{\mG}{\mathbb{G}}
\newcommand{\mZ}{\mathbb{Z}}
\newcommand{\mA}{\mathbb{A}}

\newcommand{\mC}{\mathbb{C}}
\newcommand{\mP}{\mathbb{P}}
\newcommand{\mF}{\mathbb{F}}

\newcommand{\End}{\mathrm{End}}
\newcommand{\im}{\mathrm{im}}
\newcommand{\Ob}{\mathrm{Ob}}

\newcommand{\colim}{\mathrm{colim}}
\newcommand{\soc}{\mathrm{soc}}
\newcommand{\Hom}{\mathrm{Hom}}

\newcommand{\Aut}{\mathrm{Aut}}
\newcommand{\Nat}{\mathrm{Nat}}
\newcommand{\Sym}{\mathrm{Sym}}

\newcommand{\cG}{\mathcal{G}}
\newcommand{\op}{\mathrm{op}}
\newcommand{\Ind}{\mathrm{Ind}}
\newcommand{\Res}{\mathrm{Res}}

\newcommand{\Spec}{\mathrm{Spec}}

\newcommand{\Nil}{\mathrm{Nil}}

\newcommand{\Vecc}{\mathsf{Vec}}
\newcommand{\sVec}{\mathsf{sVec}}
\newcommand{\Ver}{\mathsf{Ver}}
\newcommand{\Top}{\mathsf{Top}}

\newcommand{\Sh}{\mathsf{Sh}}

\newcommand{\bX}{\mathbf{X}}

\newcommand{\bY}{\mathbf{Y}}
\newcommand{\bZ}{\mathbf{Z}}
\newcommand{\bU}{\mathbf{U}}
\newcommand{\bSpec}{\mathbf{Spec}}
\newcommand{\unit}{{\mathbbm{1}}}
\newcommand{\cC}{\mathsf{C}}

\newcommand{\supp}{\mathrm{supp}}

\newcommand{\cM}{\mathcal{M}}

\begin{document}
\title[Geometry in tensor categories]{Algebraic geometry in tensor categories}
\author{Kevin Coulembier}
\address{K.C.: School of Mathematics and Statistics, University of Sydney, NSW 2006, Australia}
\email{kevin.coulembier@sydney.edu.au}



\begin{abstract}
We set up some foundations of generalised scheme theory related to new incompressible symmetric tensor categories. This is analogous to the relation between super schemes and the category of super vector spaces.
\end{abstract}

\maketitle


\section*{Introduction}

We initiate a systematic study of `algebraic geometry associated to a tensor category'. Following \cite{Del90, Del02}, a tensor category is an abelian rigid symmetric (contrary to the convention in \cite{EGNO}) monoidal category over a field $k$. The motivation for this paper originates in the discrepancy in complication between the theory of tensor categories over fields of characteristic zero and over fields of positive characteristic.

Concretely, in characteristic zero, Deligne proved in \cite{Del90, Del02} that two tensor categories stand out: the tensor categories of vector spaces and of super vector spaces. Via tannakian reconstruction, all tensor categories of moderate growth can be constructed out of these two tensor categories. In the language of \cite{BE, BEO, CEO2}, these two categories are the unique `incompressible' tensor categories of moderate growth. Furthermore, intuitively, one can also `generate' the theories of algebraic geometry and super algebraic geometry (over fields) out of these two tensor categories.

In sharp contrast, in positive characteristic, there are at least countably many incompressible tensor categories of moderate growth, which take over the role previously monopolised by (super) vector spaces, see \cite{BE, BEO, AbEnv, CEO, Os}. For several reasons, for instance in order to study representation theory of internal affine group schemes, it seems worthwhile to develop a notion of `algebraic geometry' associated to these tensor categories, as analogues of super geometry. Since the complete classification of the incompressible tensor categories is not currently established, and the known incompressible tensor categories still hold some mysteries, in the current paper we do not focus on them specifically. Rather we try to define scheme theory for as many tensor categories as possible. We note that the idea of developing geometries via tensor categories was also suggested by Manin \cite[\S 1.4]{Manin} in the context of noncommutative geometry.

Interestingly, there is already a notion of algebraic geometry associated to symmetric monoidal categories in much greater generality. This is the abstract scheme theory of \cite{TV}, motivated by `the field with one element $\mF_1$'. Indeed, a rigorous notion of $\mF_1$-geometry can be defined as the geometry associated via \cite{TV} to the monoidal category $\Set$. However, the theory \cite{TV} does not involve a `geometric realisation' of the schemes, contrary for instance to super schemes which are specific sheaves of super algebras on topological spaces. Indeed, the schemes in \cite{TV} are only realised as specific faisceaux (sheaves) on the opposite of the category of algebras. However, given that the motivation of \cite{TV} was to develop geometry from categories `deeper' than the category of abelian groups, and given that incompressible tensor categories are the deepest tensor categories, it is still sensible to expect the theory of \cite{TV} to be relevant. This will be confirmed in this paper.

Orthogonal to the idea of developing algebraic geometry for incompressible tensor categories is associating algebraic geometry to tannakian categories, {\it i.e.} representation categories $\Rep G$ of affine group schemes $G$. A natural desire, also expressed in \cite[\S 1.4]{Manin}, for any potential `algebraic geometry' associated to $\Rep G$ would be for it to yield $G$-equivariant geometry. In Section~\ref{SecFais}, we demonstrate that the category schemes with $G$-action embeds fully faithfully into the category of faisceaux associated to $\Rep G$, which makes it possible to study $G$-equivariant geometry in a formal way purely starting from~$\Rep G$. On the other hand, we show that the image of this embedding does {\em not} correspond to the `schemes' in the sense of~\cite{TV}. Furthermore, since schemes with $G$-action (in their geometric realisation) can also not be constructed by glueing together affine schemes with $G$-action\footnote{Infinitesimal group schemes yield the exception to all these observations, and the corresponding tannakian categories {\em will} be included in our constructions.}, this also precludes an approach to equivariant geometry in the spirit of super geometry. 

We are thus logically forced to restrict our attention to specific tensor categories. Concretely, we impose some conditions which naturally appeared in the study of commutative algebra in \cite{ComAlg} and which are conjectured to hold for the known incompressible tensor categories. This conjecture has already been verified for $p=2$ in \cite{CEO2}. Under such assumptions, we show in Section~\ref{SecGeoRel} that every algebra in the tensor category admits a geometric realisation, as a locally ringed space (internal to the tensor category). For completeness, we prove that these realisations do not always exist for tannakian categories.

In Section~\ref{SecSch}, we then define the category of schemes associated to an appropriate tensor category as the category of those locally ringed spaces which are locally geometric realisations of algebras. We derive some basic properties and, as an example demonstrating the soundness of the theory, show that the various characterisations of quasi-affine schemes from the classical case carry over. We show that the category of schemes is equivalent to a subcategory of faisceaux, via association of the functor of points. This shows in particular that we have indeed developed a geometric companion of the abstract scheme theory of \cite{TV}, for those tensor categories satisfying our assumptions.
In Section~\ref{SecAlgSch} we verify that basic essential properties of algebraic schemes (of finite type over a field) are satisfied for schemes in tensor categories.

The first major question that arises is whether quotients of algebraic groups (in appropriate tensor categories) by subgroups exist as schemes. For the category of super vector spaces this was answered in the affirmative in \cite{MZ}. We leave this question for future work. However, in Section~\ref{SecGH}, we take this existence as a hypothesis and derive several consequences, regarding observable subgroups.

Finally, in Section~\ref{SecEx}, we focus on some examples in the smallest non-semisimple incompressible tensor category: $\Ver_4^+$, in characteristic $2$. We show that a certain projective space is actually simultaneously affine; and give an example of (the existence) of a homogeneous space.

\section{Preliminaries}

\subsection{Tensor categories}

We refer to \cite{Del90, DM, EGNO} for a full introduction to the theory of tensor categories.

\subsubsection{}Following \cite{Del90}, an essentially small $k$-linear symmetric monoidal category $(\bC,\otimes,\unit)$ is a {\bf tensor category over $k$} if
\begin{enumerate}
\item $\bC$ is abelian;
\item $k\to\End(\unit)$ is an isomorphism;
\item $(\bC,\otimes,\unit)$ is rigid, meaning that every object $X$ has a monoidal dual $X^\vee$.
\end{enumerate}
By \cite[Proposition~1.17]{DM}, it then follows that $\unit$ is simple.
If additionally, we have
\begin{enumerate}
\item[(4)] Every object in~$\bC$ has finite length;
\end{enumerate}
then using (2) and (3) shows also that morphism spaces are finite dimensional. Following \cite{EGNO} we will henceforth always assume that (4) is satisfied in a tensor category. What we call tensor category is thus called pretannakian category in \cite{CEO, ComAlg}.

A {\bf tensor functor} between tensor categories is a $k$-linear exact symmetric monoidal functor. It is well-known, see \cite{DM}, that tensor functors are automatically faithful. 

\subsubsection{} In an abelian category $\bA$ that admits all (small) coproducts, a {\bf generator} is an object~$G\in\bA$ such that 
$\Hom(G,-):\bA\to\Ab$
is faithful. Equivalently, every object in $\bA$ is a quotient of a direct sum of copies of $G$.

To avoid confusion, we will not use the term generator for an object which generates a tensor category as a tensor category.

\subsubsection{} We will make little distinction between a tensor category $\bC$ and its ind-completion $\Ind\bC$. Similarly, we will use the same notation for a tensor functor and its cocontinuous extension to the ind-completions. The ind-completion of a tensor category $\cC$ is a Grothendieck category, in particular a generator is given by a direct sum over isomorphism classes of objects in $\cC$.

\subsubsection{} The most basic tensor category over $k$ is the category of finite-dimensional vector spaces $\Vecc=\Vecc_k$. Its ind-completion is the category $\Vecc^\infty$ of all vector spaces over $k$. For any tensor category $\cC$ we have an inclusion $\Vecc\hookrightarrow\cC$ and we denote its right adjoint, $\cC\to\Vecc$ or $\Ind\cC\to\Vecc^\infty$ by $H^0$, see \cite[\S 1.3]{ComAlg}. Furthermore, we typically abbreviate $H^0(X)$ to $X^0$.

\subsection{Commutative algebra} We refer to \cite{ComAlg} for a complete introduction to commutative algebra in tensor categories.

\subsubsection{}
We denote by $\Alg\cC$ the category of commutative algebras in $\Ind\cC$ and by $\LocAlg\cC$ the subcategory of local algebras and morphisms. We also write $\Aff\cC:=(\Alg\cC)^{\op}$ and refer to its objects as affine schemes. For $A\in\Alg\cC$, we have the monoidal category of $A$-modules $\Mod_{\cC}A$ in $\Ind\cC$.

For $A\in\Alg\cC$ we have the subalgebra $\Alg_k\ni A^0\subset A$. Dually, we have the maximal quotient algebra $A\tto A_0$ which belongs to $\Alg_k$. Its kernel is denoted by $R(A)<A$.

We can localise with respect to subsets $S\subset A^0$, see \cite[\S 4.1]{ComAlg}. In particular, for $f\in A^0$, we have the algebra morphism $A\to A_f$ and for an $A$-module $M$ we have the $A_f$-module $M_f$.

We recall the following simple results from \cite[5.4.1 and 5.4.2]{ComAlg}.

\begin{lemma}\label{LemfInv}
Consider $A\in\Alg\cC$.
\begin{enumerate}
\item An element $f\in A^0$ is invertible if and only if no proper ideal of $A$ contains $f$.
\item If $A$ is local with maximal ideal $\fm$, then $A^0$ is local with maximal ideal $\fm^0$. 
\end{enumerate}
\end{lemma}

\subsubsection{} For $A\in\Alg\cC$, we have the topological space $\Spec A$. For $f\in A^0$ we denote by $(\Spec A)_f$ the open subset of prime ideals $\fp<A$ for which $f\not\in \fp^0$. It then follows easily that $A\to A_f$ induces a homeomorphism
$$\Spec(A_f)\;\to\; (\Spec A)_f,$$
so that henceforth we will simply write $\Spec A_f$ for either space.

\begin{lemma}\label{LemXf0}
The following are equivalent for $f,g\in A^0$:
$$\Spec A_f\subset \Spec A_g \;\Leftrightarrow\; \Spec A^0_f\subset \Spec A^0_g.$$
\end{lemma}
\begin{proof}
That the right inclusion implies the left is obvious. On the other hand, the left inclusion implies that $g$, when interpreted in $A_f$, is not included in any prime ideal. By Lemma~\ref{LemfInv}(1), it follows that $g$ is invertible in $A^0_f$, and the right inclusion follows.
\end{proof}

\subsection{The central hypotheses}

\subsubsection{}\label{hypo} We will develop the framework of algebraic geometry for tensor categories $\cC$ satisfying the following hypothesis:

For every $A\in \Alg\cC$, the following equivalent properties are satisfied:
\begin{enumerate}
\item The map $$\rho_A:\;\Spec A\;\to\; \Spec A^0$$
is the inclusion of a subspace, for each $A\in\Alg\cC$;
\item For every ideal $I<A$, we have $V(I)=V(AI^0)$;
\item The topology on $\Spec A$ is the pullback of the topology on $\Spec A^0$ under $\rho_A$;
\item The subsets $\{\Spec A_f\,|\, f\in A^0\}$ form a basis for the topology on $\Spec A$.
\end{enumerate}

That (1) - (3) are equivalent is an easy exercise, worked out in \cite[Lemma~5.1.5]{ComAlg}. The equivalence between (3) and (4) is immediate.

Under the above hypothesis, there is a good theory of localisation at prime ideals, as detailed in \cite{ComAlg}, which we will freely use. Concretely, we set $A_{\fp}=A\otimes_{A^0}A^0_{\fp^0}$ for every prime ideal $\fp<A$.

\begin{remark}
We know of no examples where Hypothesis~\ref{hypo} is satisfied without the stronger condition, that $\rho_A$ is actually a homeomorphism, is satisfied. The latter (so in particular \ref{hypo}) is satisfied in the following two cases, see \cite[\S 5]{ComAlg}:
\begin{enumerate}
\item $\cC$ is {\bf MN} and {\bf GR};
\item $\cC$ is the representation category of an infinitesimal group scheme.
\end{enumerate}
Moreover, if $\cC$ is {\bf MN}, then it follows from \cite[Corollary~5.3.2]{ComAlg} that $\cC$ satisfies Hypothesis~\ref{hypo} if and only if $\cC$ is also {\bf GR}.
\end{remark}

\subsubsection{} For convenience we recall the definitions of {\bf MN} and {\bf GR}. 
We say that $\cC$ is {\bf maximally nilpotent}, or simply {\bf MN}, if
\begin{enumerate}
\item[{\bf(MN1)}] For every simple $L\not=\unit$ in~$\cC$, the algebra $\Sym L$ is finite; and
\item[{\bf (MN2)}] For every non-split $\alpha:\unit\hookrightarrow X$ in~$\cC$, there exists $n$ for which $\alpha^n:\unit\to\Sym^nX$ is zero.
\end{enumerate}

We say that $\cC$ is {\bf geometrically reductive}, or simply {\bf GR}, if
\begin{enumerate}
\item[{\bf(GR)}] For every non-zero morphism $X\tto\unit$ in~$\cC$, there exists $n>0$ for which $\Sym^n X\tto\unit$ is split.
\end{enumerate}

\subsubsection{}\label{ConsMNGR} Tensor categories satisfying {\bf MN} and {\bf GR} have many desirable properties, see \cite{ComAlg, CEO2}. For instance, by \cite[\S 3 and \S 6]{ComAlg}, all finitely generated algebras are noetherian, Nakayama's lemma holds, commutative Hopf algebras are faithfully flat over Hopf subalgebras and all simple algebras in $\Alg\cC$ are field extensions of the base field.

\section{Sheaves with values in a tensor category}

Let $\bC$ be a tensor category over a field $k$.

\subsection{Sheaves}
As in any Grothendieck category, we have a well-behaved notion of sheaves on a topological space with values in $\Ind\bC$. The case $\bC=\Vecc$ yields ordinary sheaves with coefficients in $k$.

\subsubsection{}Concretely, by a $\bC$-presheaf on a topological space $X$, we understand a functor
$$\cF:\;\mO(X)^{\op}\to \Ind\bC,$$
where $\mO(X)$ is the category of open subsets of $X$ and inclusions. Then $\cF$ is a $\cC$-sheaf if
$$\cF(U)\;\to\; \prod_a \cF(U_a)\;\rightrightarrows\;\prod_{a,b}\cF(U_a\cap U_b)$$
is an equaliser in $\Ind\bC$, for each open cover $U=\cup_a U_a$ of each open subset $U\subset X$. We denote the corresponding category by $\Sh(X,\cC)$. Similarly, a sheaf of $\cC$-algebras (which we will usually call a ringed space in $\cC$) is a functor from $\mO(X)^{\op}$ to $\Alg\cC$, where the equaliser condition can be interpreted either in $\Alg\cC$ or $\Ind\cC$. 

\subsubsection{}As usual, the stalk of a $\cC$-sheaf $\cF$ at $x\in X$ is given by
$$\cF_x:=\varinjlim_{U\ni x}\cF(U)\;\in\Ind\cC.$$

Let $\cF$ be a $\bC$-sheaf on a topological space $X$. Since $H^0$ is continuous, the presheaf $\cF^0$ defined by $U\mapsto \cF(U)^0$ is a $k$-valued sheaf on $X$. Moreover, since $H^0$ commutes with direct limits, we have isomorphisms of stalks $\cF^0_x\simeq (\cF_x)^0$ for all $x\in X$.

For a $\cC$-sheaf $\cF$ on $X$, its {\bf support} $\supp\cF\subset X$ consists of all $x\in X$ with $\cF_x\not=0$.

\subsubsection{}
We will freely use classical notions such as the pushforward $f_\ast\cF$ and inverse image sheaf $f^{-1}\cG$, for $f:X\to Y$ a map (= continuous function) of topological spaces and $\cF$, resp.~$\cG$, a sheaf on $X$, resp. $Y$. 
For any $\bC$-sheaf $\cF$ on $X$, we have the object of global sections 
$$\Gamma(X,\cF)=\cF(X)\in\Ind\cC.$$

\subsubsection{}\label{Bsheaf} For a basis $B$ which is closed under finite intersection (this condition can easily be omitted) of a topological space, we can define $\cC$-sheaves with respect to $B$ as functors $\cF:B^{\op}\to\Ind\cC$ for which
$$\cF(U)\;\to\; \prod_a \cF(U_a)\;\rightrightarrows\;\prod_{a,b}\cF(U_a\cap U_b)$$
is an equaliser in $\Ind\cC$,
for every $U\in B$ covered by $\{U_a\in B\}$. By \cite[Proposition~I.12]{EH}, the functor which restricts a sheaf on $X$ to $B\subset \mO(X)$ yields an equivalence of categories.

\begin{corollary}\label{CorBsh} Consider $A\in\Alg\cC$.
If $\Spec A\to \Spec A^0$ is the inclusion of a subspace, then pullback yields an equivalence
$$\Sh(\Spec A^0,\cC)\;\xrightarrow{\sim}\; \Sh(\Spec A,\cC).$$
Moreover, if a sheaf $\cF$ on $\Spec A^0$ is sent to $\cG$, then canonically $\cG(\Spec A_f)\simeq\cF(\Spec A^0_f)$ for all $f\in A^0$.
\end{corollary}
\begin{proof}
By assumption, $\{\Spec A_f\,|\, f\in A^0\}$ forms a basis of $\Spec A$. By Lemma~\ref{LemXf0}, these bases for $\Spec A$ and $\Spec A^0$ yield isomorphic categories. The conclusion thus follows from~\ref{Bsheaf}.
\end{proof}

\subsection{Locally ringed spaces}

\subsubsection{}\label{DefsLRS} By a locally ringed space in $\cC$ we understand a pair $\bX=(X,\cO)$ of a topological space $X$ and a sheaf of $\bC$-algebras $\cO$ on $X$ such that each stalk $\cO_x$ is a local algebra. A morphism of locally ringed spaces $(X,\cO_X)\to (Y,\cO_Y)$ is a pair $f:X\to Y$, $f^\sharp:\cO_Y\to f_\ast\cO_X$ (a morphism of ringed spaces), for which the canonical composite morphism
$$\cO_{Y,f(x)}\;\to\;\varinjlim_{Y\supset U\ni f(x)}\cO_X(f^{-1}(U))\;\to\;\cO_{X,x}$$
is local for every $x\in X$.
 We have the corresponding category $\mathsf{Lrs}\bC$ of locally ringed spaces. 

For $\cC=\Vecc_k$, we abbreviate $\LRS\cC$ to $\LRS_k$. 

\begin{lemma}\label{LemX0}
The canonical inclusion $\LRS_k\to \LRS\cC$ 
\begin{enumerate}
\item has a left adjoint, $\bX=(X,\cO)\mapsto \bX^0:=(X,\cO^0)$; and
\item has a right adjoint, $\bX=(X,\cO)\mapsto\bX_0:=(X',\cO')$ where $\cO'$ is the sheafification of $U\mapsto \cO(U)_0$ and $X'=\supp \cO'\subset X$.
\end{enumerate} 
\end{lemma}
\begin{proof}
The only thing to verify for part (1) is that $\bX^0$ is a {\em locally} ringed space and that relevant morphisms to $\bX^0$ are local. These are direct consequences of Lemma~\ref{LemfInv}(2).

Part (2) is a standard exercise.
\end{proof}

\subsubsection{}\label{DefXf} Consider $\bX=(X,\cO)\in\LRS\cC$. For $f\in \Gamma(X,\cO)^0$, we define 
\begin{eqnarray*}X_f&=&\{x\in X\,|\,\mbox{$f$ is not in the maximal ideal of $\cO_{x}$}\}\\
&=&\{x\in X\,|\,\mbox{$f$ is invertible in $\cO_{x}^0$}\},
\end{eqnarray*}
where equality of the two characterisations follows from Lemma~\ref{LemfInv}(1).
This is an open subset of $X$, which follows for instance from the classical case by considering $\bX^0$. 

We denote by $\bX_f$ the locally ringed space $(X_f,\cO|_{X_f})$.

\begin{remark}
By adapting the first line of the definition of $X_f$, we can define more generally an open subspace of $X$ for every $\cC\ni V\subset \Gamma(X,\cO)$.
\end{remark}

\subsubsection{} For a locally ringed space $\bX=(X,\cO)$, we can also define in a straightforward way the (abelian) category of $(X,\cO)$-modules in $\Ind\cC$. For $N\in\Ind\cC$ and an $(X,\cO)$-module $\cM$, we get a canonical isomorphism
\begin{equation}\label{GlobSec}\Hom_{\cO}(\cO\otimes N,\cM)\;\xrightarrow{\sim}\;\Hom(N,\Gamma(X,\cM)).\end{equation}

An $(X,\cO)$-module $\cM$ is {\bf quasi-coherent} if for every $x\in X$ there is an open neightbourhood $x\in U\subset X$ for which there exists a presentation
$$\cO|_U\otimes M_1\to \cO|_U\otimes M_0\to \cM|_U\to 0$$
of $\cO|_U$-modules, for certain $M_0,M_1\in\Ind\bC$. We denote the corresponding full subcategory of quasi-coherent sheaves by $\QCoh_{\cC}\bX$.

\begin{prop}\label{LemQCoh1}
\begin{enumerate}
\item For a $k$-scheme $\bX$, interpreted as an object in $\LRS_k\subset\LRS\cC$, an $\bX$-module is quasi-coherent if and only if there exists an open affine cover such that for each $U\subset X$ in the cover and $f\in \cO_{\bX}(U)$, the restriction map yields an isomorphism $\cM(U)_f\to\cM(U_f)$.
\item For an affine $k$-scheme $\bX$ with algebra of global sections $R$,
$$\Gamma: \QCoh_{\bC}\bX\;\to\; \Mod_{\bC}R$$
is an equivalence.
\end{enumerate}
\end{prop}

\begin{proof}
Part (1) follows from exactness of localisation.

We have the basis $\{\Spec R_f\,|\, f\in R\}$ for the topology on $X=\Spec R$. By \ref{Bsheaf}, from $M\in\Mod_{\cC}R$ we can define a sheaf $\widetilde{M}$ by specifying $\widetilde{M}(U_f)\;=\; M_f$, with same restriction maps as in the classical case. As in the classical case, see for instance \cite[Chapter I]{EH}, proving that this construction produces a sheaf can be reduced to proving that, for a finite set $\{f_a\}\subset R$ which generates $R$ (as an ideal), we have an equaliser
$$M\to\prod_a M_{f_a}\rightrightarrows \prod_{a,b} M_{f_af_b}.$$

If $\bC\ni X\subset M$ is sent to zero by the first map, then (using $M_f=\varinjlim M$, for the direct limit along $f:M\to M$) we find for each $a$ some $i_a\in\mN$ for which $f_a^{i_a}( X)=0$. Since $\{f_a^{i_a}\,|\,a\}$ still generates $R$, this implies $X=0$.

Next, consider a morphism $\bC\ni X\to \prod_a M_{f_a}$ for which the two compositions with the morphisms to $\prod_{a,b} M_{f_af_b}$ become equal. By finiteness of $\{a\}$, we know there exists $N\in\mN$ for which each $X\to M_{f_a}$ is induced from some $j_a:X\to M$, with $M$ placed in the $N$-th position in the chain defining $M_{f_a}=\varinjlim M$. By replacing $N$ by a larger number if necessary, the equalising property implies that $f_a^N\circ j_b=f_b^N\circ j_a$. There exist $e_a\in R$ with $1=\sum_a e_a f_a^N$. It now follows easily that $X\to \prod_a M_{f_a}$ factors through $j:= \sum_a e_a j_a: X\to M$, concluding the proof that $\widetilde{M}$ is a sheaf.

The association $M\mapsto \widetilde{M}$ provides an inverse to $\Gamma$. Indeed, $\Gamma(\widetilde{M})\simeq M$ is by definition, and $\widetilde{\Gamma(\cM)}\simeq \cM$ follows from part (1).
\end{proof}


\section{Faisceaux and equivariant geometry}\label{SecFais}

\subsection{Grothendieck topologies}

\subsubsection{}

As in the classical case, we can equip the category $\Aff\cC:=(\Alg\cC)^{\op}$ with various Grothendieck (pre)topologies (in order to work with an essentially small category, we can restrict the sizes of the allowed algebras). We will only mention and use two in the current paper. We denote by $\Fun\cC$ the category of presheaves on $\Aff\cC$, meaning functors $\Alg\cC\to\Set$. We abbreviate $\Fun_k=\Fun\Vecc_k$. 

For a pretopology $T$ on $\Aff\cC$, we denote by 
$$\Fais_T\cC \subset \Fun\cC$$
the category of $T$-sheaves, which we will call $T$-faisceaux to distinguish from sheaves on topological spaces with values in~$\Ind\cC$. In case of the fpqc pretopology below, we leave out reference to $T$.

\begin{definition}
The covers in the {\bf fpqc pretopology} (`flat topology' in the terminology of the more general setting in \cite{TV}) on $\Aff\cC$ correspond to the finite collections of algebra morphisms $\{A\to B_i\,|\, i\in I\}$ for which $A\to \prod_i B_i$ is faithfully flat.
\end{definition} 

Hence, $F\in \Fun\cC$ belongs to $\Fais\cC$ if and only if $F$ commutes with finite products and if 
$$F(A)\to F(B)\rightrightarrows F(B\otimes_AB)$$
is an equaliser for every faithfully flat algebra morphism $A\to B$.


\subsubsection{}\label{Zar} We can impose the following conditions on $A\to B$ in~$\Alg\cC$:
\begin{enumerate}
\item $A\to B$ is an epimorphism in~$\Alg\cC$;
\item $A\to B$ is flat;
\item $A\to B$ is finitely presented.
\end{enumerate}  
The obvious example of a morphism satisfying all three conditions is $A\to A_f$ for some $f\in A^0$.

Following \cite{TV} we can then define the following pretopology:
\begin{definition}
The covers in the {\bf Zariski pretopology} on $\Aff\cC$ correspond to the finite collections of algebra morphisms $\{A\to B_i\,|\, i\in I\}$ for which $A\to \prod_i B_i$ is faithfully flat and each $A\to B_i$ satisfies \ref{Zar}(1)-(3).
\end{definition}

We denote the category of sheaves for the Zariski pretopology by $\Fais_z\cC$ and observe that, by definition, it contains $\Fais\cC$.
In convenient settings the Zariski pretopology reduces to a more simplistic one. Note that for a collection of elements $f_i\in A^0$, the morphism $A\to \prod_iA_{f_i}$ is faithfully flat if and only if the $f_i$ generate $A$ (as an $A$-module), by \cite[Proposition~4.3.1]{ComAlg}.
\begin{prop}\label{PropTV}
Under hypothesis~\ref{hypo}, the (Grothendieck) topology generated by the Zariski pretopology is the same as the topology generated by the pretopology corresponding to finite collections $\{A\to A_{f_i}\}$ for which the $f_i\in A^0$ generate $A$.
\end{prop}
\begin{proof}
If suffices to prove that every covering in the Zariski pretopology can be refined to one of the form $\{A\to A_{f_i}\}$.

We consider a morphism $A\to B$ satisfying the conditions in \ref{Zar}.
Let $S\subset A^0$ be the subset of $f\in A^0$ for which the induced
$$A_f\to B_f$$
(where we also write `$f$' for its image in~$B^0$) is an isomorphism.
We claim that the image of~$S$ in~$B^0$ generates $B$. The conclusion then follows easily.

Assume the contrary of the claim. To obtain a contradiction we will freely use \cite[Proposition~5.4.4]{ComAlg} concerning localisation at prime ideals. As the image of $S$ does not generate $B$, there exists a maximal ideal $\fm<B$ with preimage $\p<A$ such that $S\subset \p^0$. Now the induced local morphism
$$A_{\p}\to B_{\fm}$$
is still an epimorphism (since $B\to B_{\fm}$ is an epimorphism). It is also still flat, so by \cite[Proposition~4.3.1]{ComAlg} faithfully flat. By \cite[Remark~2.3.4]{ComAlg}, the morphism is thus an isomorphism. We can then draw the isomorphism and its inverse in the following solid commutative diagram
$$\xymatrix{
A_{\fp}\ar[r]^{\sim}& B_{\fm}\ar[r]^{\sim}& A_{\fp}\\
A\ar[u]\ar[r]&B\ar[u]\ar@{-->}[r]&A_f.\ar[u]
}$$
Since $A\to B$ is finitely presented, `$A_{\fp}=\varinjlim A_f$' shows that there must be some $f\in A^0\backslash \fp^0$ so that $B\to A_{\fp}$ factors as above (and the composite of the lower row is the localisation $A\to A_f$). Localising the lower row thus gives morphisms
$$A_f\to B_f\to A_f$$
 which compose to the identity. Using that $A_f\to B_f$ is an epimorphism shows that also $B_f\to A_f\to B_f$ composes to the identity. But by assumption $f\not\in\fp^0\supset S$, a contradiction.
\end{proof}

\begin{lemma}\label{LemTech} Let $\cC$ be a tensor category satisfying Hypothesis~\ref{hypo}.
Let $\phi:F\to G$ be a monomorphism $\Fais_z\cC$ which evaluates to an isomorphism on every local algebra. Assume that $F$ commutes with direct limits and for every directed system $A_\alpha$ in $\Alg\cC$, the map $\varinjlim_\alpha G(A_\alpha)\to G(\varinjlim_\alpha A_\alpha)$ is injective. Then $\phi$ is an isomorphism. 
\end{lemma}
\begin{proof}
We consider $A\in\Alg\cC$ and the map
$$\phi_A:\; F(A)\to G(A).$$
We fix $x\in G(A)$. For every maximal ideal $\fm<A$, we can consider the image $x_{\fm}$ of $x$ in $G(A_{\fm})$. By assumption, $x_{\fm}$ is the image of an element in $F(A_{\fm})$. Also by assumption, this element is in turn the image of an element in $F(A_f)$, for some $\fm^0\not\ni f\in A^0$. We can consider the image $x_f$ of the latter in $G(A_f)$. By construction, this $x_f$ goes to $x_{\fm}$ in $G(A_{\fm})$, hence, by assumption, the images of $x_f$ and $x$ in some $G(A_g)$ (with $\fm^0\not\ni g\in A^0$) must be equal, and be in the image of $\phi_{A_g}$. 

The above yields in particular a (non-unique) assignment $\fm\mapsto g_{\fm}$ from maximal ideals in $A$ to elements of $A^0$, with $g_{\fm}\not\in \fm^0$. Hence there is no maximal ideal in $A$ which contains all $\{g_{\fm}\}$. The same is thus true for a finite subset of $\{g_{\fm}\}$. This yields a Zariski cover $\{A\to A_{g_{\fm}}\}$. The conclusion thus follows from \cite[Lemma~2.4.3]{ComAlg}.
\end{proof}

\subsection{Example: Projective space}

We fix $X\in\cC$.

\begin{definition}
The functor $\mP_X\in\Fun\cC$ is given by
$$A\mapsto \mP_X(A)\,=\,\{A\otimes X^\ast\tto\cL\}/\simeq$$
where $\cL$ runs over invertible $A$-modules and the equivalence relation is given by isomorphisms of invertible $A$-modules yielding commutative triangles.
\end{definition}

\begin{remark}\label{RemPX}
We can equivalently describe $\mP_X(A)$ as the collection of submodules $M\subset A\otimes X^\ast$ for which $(A\otimes X^\ast)/M$ is invertible, or the collection of invertible submodules $\cL\subset A\otimes X$ for which $(A\otimes X)/\cL$ is rigid.
\end{remark}

\begin{example}
We have
$$\mP_X(\unit)\;=\; \bigsqcup_{L}\mP(\Hom(L,X)),$$
where the index runs over isomorphism classes of invertible objects $L$ in $\cC$, and $\mP(V)$ is the ordinary projective space associated to a $k$-vector space $V$.
\end{example}

\begin{prop}
We have $\mP_X\in\Fais\cC$.
\end{prop}
\begin{proof}
That $\mP_X$ commutes with finite products is clear. For a faithfully flat algebra morphism $A\to B$, the induced function
$\mP_X(A)\to\mP_X(B)$ is injective, by interpretation of $\mP_X(-)$ as in Remark~\ref{RemPX}.

Finally, an element in $\mP_X(B)$ which is sent to the same element in $\mP_X(B\otimes_AB)$ under both $B\rightrightarrows B\otimes_AB$ is represented by an invertible $B$-module $\cL$ equipped with $p:B\otimes X^\ast\tto \cL$ and an isomorphism as in the top row of
$$\xymatrix{B\otimes_A\cL\ar[rr]^{\sim}&&B\otimes_A\cL\\
B\otimes_AB\otimes X^\ast\ar@{->>}[u]^{B\otimes_Ap}\ar[rr]^{\sim}&&B\otimes_AB\otimes X^\ast\ar@{->>}[u]^{B\otimes_Ap} ,}$$
which exchanges the two $B\otimes_AB$-module structures (related by the flip map). Moreover the square is commutative where the lower horizontal arrow is given by the flip map on $B\otimes_AB$. 

We need to show that the top isomorphism satisfies the cocycle condition and that $p$ induces a morphism of descent data, for the obvious descent data on $B\otimes X^\ast$. As $p$ is an epimorphism, both conditions follow automatically from the commutative square.

It now follows from faithfully flat descent, see \cite[\S 4.10]{Brand} that $\cL\simeq B\otimes_A\cL'$ for some invertible (see \cite[Proposition~4.10.3]{Brand}) $A$-module $\cL'$ and $p$ corresponds to $B\otimes_Ap'$ for some $p':A\otimes X\tto\cL'$.
\end{proof}


\subsection{Open covers}

\subsubsection{}
For an algebra $A\in\Alg\cC$ we denote the corresponding representable functor by
$$\Phi_A=\Alg\cC(A,-)\in \Fun\cC. $$
For an ideal $I<A$ we have the subfunctor $\Phi^I_A\subset \Phi_A$. Here $\Phi_A^I(B)$ comprises those algebra morphisms $\phi:A\to B$ such that $B\phi(I)=B$, or equivalently $\Spec B\to\Spec A$ takes values in the complement of $V(I)$.

Note that $\Phi_A$, and more generally $\Phi_A^I$, belongs to $\Fais\cC$, by \cite[Remark~2.3.4]{ComAlg} and \cite[Proposition~4.3.1]{ComAlg}.

\subsubsection{}\label{OpenFun} For $F\in\Fun\cC$, a subfunctor $G\subset F$ is called {\bf open} if for each $A\in \Alg\cC$ and element of~$F(A)$, the corresponding subfunctor $G\times_F\Phi_A$ of $\Phi_A$ is of the form $\Phi_A^I$ for some ideal $I<A$. A direct computation shows that the open subfunctors of the representable functors $\Phi_A$ are precisely the subfunctors $\Phi_A^I\subset \Phi_A$.

\begin{remark}
\label{RemOpen}
\begin{enumerate}
\item For an ideal $I<A$ which is generated by some $f\in A^0$, the subfunctor $\Phi^I_A$ corresponds to $\Phi_{A_f}\to \Phi_A$.
\item More generally, under Hypothesis~\ref{hypo}, formulation~\ref{hypo}(2) implies that the subfunctors $\Phi^I_A$ are the images, computed in $\Fais\cC$ or $\Fais_z\cC$, of the maps
$$\bigsqcup_{i\in I}\Phi_{A_{f_i}}\;\to\;\Phi_A,$$
for families $\{f_i\in A^0\mid i\in I\}$.
\end{enumerate}
\end{remark}

\subsubsection{}\label{DefOC} For $F\in\Fun\cC$, a collection of open subfunctors $G_i\subset F$ is an {\bf open cover} if for each $A\in \Alg\cC$ and element of $F(A)$, the corresponding subfunctors $$G_i\times_F\Phi_A\simeq \Phi_A^{I_i}\;\subset\;\Phi_A$$ satisfy $\sum_i I_i=A$.

Note that we can clearly have $\Phi_A^I=\Phi_A^J$ for ideals $I\not=J$, in particular $\Phi_A^I=\Phi_A^J$ whenever $\sqrt{I}=\sqrt{J}$. However, the condition $\sum_i I_i=A$, from the previous paragraph, is equivalent to the condition $\cap_i V(I_i)=\varnothing$, and thus independent of the choices of $I_i$. Of course, we can also use the reformulation that the open complements of $V(I_i)$ form an open cover of $\Spec A$.

%
%

\begin{lemma}\label{LemSCover}
A collection of open subfunctors $G_i\subset F$ is an open cover if and only if for each simple algebra $S\in\Alg\cC$, the function
$$\bigsqcup_i G_i(S)\;\to\; F(S)$$
is surjective.
\end{lemma}
\begin{proof}
Assume that the $G_i$ form an open cover, then applying the definition to an arbitrary $a\in F(S)$ for simple $S$ shows the surjectivity, since the only ideals in $S$ are $0$ and $S$. 

On the other hand, assume that the $G_i$ do not form an open cover. Then there exists $A\in\Alg\cC$ for which the corresponding ideals $I_i<A$ are all contained in a maximal ideal $\fm<A$. Considering the simple algebra $S:=A/\fm$ then shows that also the surjectivity condition in the lemma fails.
\end{proof}

\begin{prop}\label{PropCover} Consider a collection of subfunctors $\{G_i\subset F\mid i\in L\}$ in $\Fun\cC$.
\begin{enumerate}
\item Assume that for every $A\in\Alg\cC$ and every element $\alpha\in F(A)$, there exists a
finite subset $L_0\subset L$ and $\{f_i\in A^0\mid i\in L_0\}$ with $1=\sum_if_i$ so that the image of $\alpha$ under $F(A)\to F(A_{f_i})$ is in $G_i(A_{f_i})\subset F(A_{f_i})$ for every $i\in L_0$. Then $\{G_i\subset F\}$ is an open cover of $F$.
\item Under Hypothesis~\ref{hypo}, the condition in (1) is also a necessary condition for $\{G_i\subset F\}$ to be an open covering.
\end{enumerate}
\end{prop}
\begin{proof}
For part (1), we can apply the assumption to the case where $A$ is simple, and hence $A^0$ is a field. The assumption then simply states that the equivalent criterion from Lemma~\ref{LemSCover} is satisfied.

For part (2), consider an open cover $\{G_i\subset F\mid i\in L\}$ and $\alpha\in F(A)$. By definition this leads to a collection of ideals $I_i<A$ with $\cap_i V(I_i)=\varnothing$. By \ref{hypo}(2), it follows that $\sum_i I_i^0=A^0$. Hence there exists a finite subset $L_0\subset L$ and $f_i\in I^0_i$ for $i\in L_0$ with $1=\sum_if_i$. Since $\Phi_{A_{f_i}}\to \Phi_A\to F$ factors via $\Phi_A^{I_i}\to G_i\to F$, the conclusion follows.
\end{proof}

\subsection{Equivariant geometry}

Let $G$ be an affine group scheme over $k$ and set $\cC=\Rep G$. 

\subsubsection{Notation}We denote the forgetful functor by
$$\omega:\; \Rep G\;\to\; \Vecc.$$
We fix the notation $\mathrm{O}(G)\in\Alg\cC$ for the algebra $\cO(G)$ with canonical left coaction, in particular $\omega(\mathrm{O}(G))=\cO(G)$.

We denote by $\Sch^G_k$ the category of pairs $(\bX,a)$ of $k$-schemes $\bX$ equipped with a $G$-action
$$a:G\times\bX\to\bX,$$
which we mostly suppress from notation.

We demonstrate that $\Fais\cC$ contains $\Sch_k^G$ fully faithfully.

\begin{theorem}\label{ThmEquiv}
The functor
$$\Sch^G_k\;\to\; \Fun\cC,\quad \bX\mapsto H_{\bX}:=\Hom_G(\underline{\Spec}-,\bX)$$
takes values in~$\Fais\cC$ and is fully faithful. Here $\underline{\Spec}A$ stands for the affine $k$-scheme corresponding to $\omega(A)\in \Alg_k$, for $A\in \Alg\cC$, equipped with canonical $G$-action.
\end{theorem}
\begin{proof}
Fix $(\bX,a)\in \Sch_k^G$. We will also write $\bX$ for the associated functor of points $\Alg_k\to\Set$ (in particular forgetting the $G$-action). For each $A\in \Alg\cC$, the diagram
\begin{equation}\label{eqr1}H_{\bX}(A)\to \bX(\omega A)
\rightrightarrows\bX(\omega A\otimes \cO(G))\end{equation}
is an equaliser. Here, one of the parallel arrows is $\bX(-)$ of the coaction of $A$ and the other is induced from $a$ (it sends $\phi:\bSpec\omega A\to\bX $ to $a\circ(G\times \phi)$).
It then follows from Lemma~\ref{TanZar}(1) and (2) below, the classical fact that $\bX(-)\in \Fais_k$ and standard diagram chasing that $H_{\bX}\in \Fais\cC$.

Now, for $\bX,\bY\in \Sch_k^G$, we consider morphisms
\begin{equation}\label{HGH}
\Hom_G(\bX,\bY)\;\to\; \Nat(H_{\bX}, H_{\bY})\;\to\; \Hom(\bX,\bY).
\end{equation}
The first morphisms is the obvious one (which we need to show is an isomorphism) and the second evaluates a natural transformation $H_{\bX}\Rightarrow H_{\bY}$ on the algebras of the form $\mathrm{O}(G)\otimes R$ for $R\in \Alg_k$, using the natural isomorphism
\begin{equation}\label{73}
H_{\bX}(\mathrm{O}(G)\otimes R)\;\simeq\; \bX(R).
\end{equation}

It then follows immediately that the composite of the two morphisms in \eqref{HGH} is the ordinary inclusion. Next we observe that the second morphism in \eqref{HGH} is injective. Indeed, we can write the injection $H_{\bX}(A)\hookrightarrow \bX(\omega A)$ from \eqref{eqr1} as the composite of isomorphism~\eqref{73}, for $R=\omega A$, and $H_{\bX}(\rho)$, for the coaction
$$\rho: A\to \mathrm{O}(G)\otimes \omega(A).$$ In particular $H_{\bX}(\rho)$ is injective. So if two natural transformations agree on all $\mathrm{O}(G)\otimes R$, $R\in\Alg_k$, they must be equal.

To conclude the proof it now suffices to show that the second morphism in \eqref{HGH} actually takes values in~$G$-equivariant morphisms. For this we can observe that for any $R\in \Alg_k$ and $g\in G(R)$, we have a commutative square
$$\xymatrix{
H_{\bX}(\mathrm{O}(G)\otimes R)\ar[rr]^{\sim}\ar[d]&& \bX(R)\ar[d]\\
H_{\bX}(\mathrm{O}(G)\otimes R)\ar[rr]^{\sim}&& \bX(R),
}$$
where the horizontal maps are \eqref{73}, the right vertical map corresponds to the action of $g$ and the left horizontal arrow represents the action of $H_{\bX}$ on the morphism in~$\Alg\cC$
$$\mathrm{O}(G)\otimes R\to \mathrm{O}(G)\otimes \cO(G)\otimes R\to \mathrm{O}(G)\otimes R\otimes R\to \mathrm{O}(G)\otimes R,$$
coming from comultiplication, $g$ and multiplication. We can thus conclude that the $G$-equivariance of $\bX\to \bY$ follows from naturality of $H_{\bX}\to H_{\bY}$. 
\end{proof}

We show how to interpret the scheme theory associated to $\Rep G$ in~\cite{TV} in terms of ordinary equivariant geometry.

\begin{theorem}\label{ThmTV}
The category of `schemes' in $\Fais_z\cC\subset\Fun\cC$, according to the definition in \cite[2.15]{TV}, is contained in the full subcategory $\Sch_k^G$ from Theorem~\ref{ThmEquiv}, and comprises those $k$-schemes with $G$-action that admit a $G$-equivariant open affine cover.
\end{theorem}

The technical observations underlying the theorem are summarised in the following lemma.

\begin{lemma}\label{TanZar}
Consider $f:A\to B$ in~$\Alg\cC$.
\begin{enumerate}
\item $f$ is flat if and only if $\omega(f)$ is flat.
\item $f$ is faithful if and only if $\omega(f)$ is faithful.
\item $f$ is an epimorphism if and only if $\omega(f)$ is an epimorphism.
\item $f$ is finitely presented if and only if $\omega(f)$ if finitely presented.
\item $f$ satisfies the conditions in \ref{Zar} if and only if the morphism of affine $k$-schemes corresponding to $\omega(f)$ is an open immersion.
\end{enumerate}
\end{lemma}
\begin{proof}
One direction of parts (1), (2) and (3) is obvious. To prove the other direction, we can use the algebra isomorphisms, natural in $A\in\Alg\cC$,
$$\mathrm{O}(G)\otimes A\;\xrightarrow{\sim} \mathrm{O}(G)\otimes \omega(A).$$
 Indeed, since $\unit \to \mathrm{O}(G)$ is obviously faithfully flat, parts (1) and (2) follow. Similarly, for (3), if $f$ is an epimorphism, it follows that $\mathrm{O}(G)\otimes \omega{f}$ is an epimorphism, and using faithful flatness of $\unit \to \mathrm{O}(G)$ allows us to conclude that if two morphisms $\omega B\rightrightarrows R$ in $\Alg_k$ are equalised by $\omega (f)$, they must be equal.

Part (4) is clear. For part (5), by (1)-(3), the conditions in \ref{Zar} are satisfied for $f$ if and only if they are satisfied for $\omega(f)$. It is well-known (and a self-contained proof of one direction is actually given in the proof of Proposition~\ref{PropTV} above applied to $\cC=\Vecc$) that these conditions characterise open immersions.
\end{proof}

\begin{proof}[Proof of Theorem~\ref{ThmTV}]
We sketch the proof:
We consider the functor
\begin{equation}\label{ForG}\Fun\cC\to\Fun_k\;:\; F\mapsto F(\mathrm{O}(G)\otimes-),\end{equation}
which satisfies $\Phi_A\mapsto \Phi_{\omega A}$.

By Remark~\ref{RemOpen}(2) and Lemma~\ref{TanZar}(5), the notion of `Zariski open' of $\Phi_A$ in \cite[D\'efinition~2.12]{TV} corresponds precisely to that of open subfunctors $\Phi_A^I\subset\Phi_A$ from \ref{OpenFun}. In particular, under~\eqref{ForG}, these Zariski opens are just sent to open subfunctors in the usual sense (corresponding to $G$-stable ideals of $\omega A$).

It then follows quickly that a scheme in the sense of \cite[D\'efinition~2.15]{TV} is sent by \eqref{ForG} to a an ordinary scheme, but by construction the affine cover can be taken $G$-equivariantly.

Conversely, starting from a scheme with $G$-action and a $G$-equivariant affine cover, we can (by the above arguments) construct the required morphisms in \cite[D\'efinition~2.15]{TV}.
\end{proof}

\begin{remark}
Theorem~\ref{ThmTV} demonstrates that the theory of schemes from \cite{TV} does not encompass all $G$-equivariant schemes. Consider for instance~$G/H$ for a subgroup $H<G$ of an algebraic group for which $k$-scheme $G/H$ is not affine.

An exception to this observation is the case of infinitesimal $G$. In this case, Theorem~\ref{ThmTV} actually implies that the theory of schemes from \cite{TV} coincides with the theory $G$-equivariant schemes. This also follows from Example~\ref{ExInf} and Remark~\ref{RemTV2} below.

\end{remark}


\section{Geometric realisations}\label{SecGeoRel}

\subsection{Definitions} 
\subsubsection{}
For $A\in\Alg\cC$, a {\bf geometric realisation} is some $\bX=(X,\cO)\in \LRS\cC$ with an algebra morphism $A\to \Gamma(X,\cO)$ that induces an isomorphism
\begin{equation}\label{eqRea}
\LRS\cC(\bY,\bX)\;\xrightarrow{\sim}\; \Alg\cC(A,\Gamma(Y,\cO_Y)),
\end{equation}
for every $\bY=(Y,\cO_Y)\in\LRS\cC$.
Thus, an algebra $A$ has a (necessarily unique up to isomorphism) geometric realisation if and only if 
$$\Alg\cC(A,\Gamma-)\;:\;(\LRS\cC)^{\op}\to \Set$$
is representable.

We will show that geometric realisations need not exist. However, when the geometric realisation $(X,\cO)$ of $A$ exists, then
there is a canonical function 
\begin{equation}\label{XA}
t_A:X\;\to\;\Spec A,
\end{equation} 
sending $x\in X$ to the preimage of the unique maximal ideal under $A\to \Gamma(X,\cO)\to\cO_x$. 


We refer to \cite{Di}, or references therein, for the definition of multi-representability.
\begin{lemma}\label{LemMR}
Assume that $(X,\cO)$ is a geometric realisation of $A\in\Alg\cC$.
\begin{enumerate}
\item With $U:\LocAlg\cC\to\Alg\cC$ the forgetful (inclusion) functor, the functor
$$\Alg\cC(A,U-):\LocAlg\cC\to \Set$$
is multi-representable, with labelling set $X$. In particular, $X$ is the geometric spectrum from \cite[\S 4.5]{ComAlg}.
\item The function $t_A$ in~\eqref{XA} has as image the set of visible prime ideals (see \cite[\S 4.2]{ComAlg}), and, under the identification in part (1), $t_A$ corresponds to \cite[equation~(4.4)]{ComAlg}.
\end{enumerate}
\end{lemma}
\begin{proof}
For part (1), associate to every local algebra the one-point locally ringed space. It then follows by existence of the geometric realisation that 
$$\Alg\cC(A,U-)\;\simeq\;\bigsqcup_{x\in X}\LocAlg\cC(\cO_x,-).$$
This proves part (1). Part (2) is immediate.
\end{proof}

\begin{example}\label{ExGa}
Set $k=\mC$ and $\cC=\Rep\mG_a$ and consider $A:=\Sym V$ with $V$ the 2-dimensional indecomposable $\mG_a$-representation. The algebra $A$ does {\bf not} have a geometric realisation. Indeed, it is proved in \cite[Example~4.4.4 and Lemma~4.5.2(2)]{ComAlg} that $\LocAlg\cC(A,U-)$ is not multi-representable, so we can apply Lemma~\ref{LemMR}.

\end{example}

\subsubsection{} Clearly the condition that every algebra in $\cC$ has a geometric realisation is equivalent to the condition that
\begin{equation}\label{eqGamma}
\Gamma:(\LRS\cC)^{\op}\;\to\; \Alg\cC
\end{equation}
has a left adjoint, or equivalently, that
\begin{equation}\label{eqGamma2}
\Gamma:(\LRS\cC)\;\to\; \Aff\cC
\end{equation}
has a right adjoint. These conditions are not automatic, by Example~\ref{ExGa}.

\subsection{Existence}


\begin{theorem}\label{ThmExistII}
Under Hypothesis~\ref{hypo}, the functor \eqref{eqGamma} has a left adjoint
$$\bSpec\;: \;\Alg\cC\to (\LRS\cC)^{\op}.$$
Moreover, for $(X,\cO):=\bSpec A$, we have that $t_A:X\to\Spec A$ in \eqref{XA} is a homeomorphism and $\eta_A:A\to \Gamma(X,\cO)$ is an isomorphism. In particular, $\bSpec$ is fully faithful.
\end{theorem}
\begin{proof}
First we can consider $A$ as an object in $\Mod_{\bC}A^0$. The proof of Proposition~\ref{LemQCoh1} yields a sheaf (of algebras) $\widetilde{A}$ on $\Spec A^0$, with global sections $A$. We define $(X,\cO)=\bSpec A:=\rho_A^{-1}\widetilde{A}$, the inverse image sheaf on the subspace $\Spec A\subset \Spec A^0$, see Corollary~\ref{CorBsh}. In particular, we have $\cO(\Spec A_f)\simeq A_f$, so $\Gamma(\bSpec A)$ is indeed $A$.

That $(X,\cO)$ is a {\em locally} ringed space follows from \cite[Proposition~5.4.4]{ComAlg} and its proof. Indeed, for $\fp\in\Spec A$, by construction, we find that $\cO_{\fp}$ is the localisation of $A$ at the multiplicative set $A^0\backslash \fp^0$, which is proved to be local {\it loc. cit.}

Proving isomorphism~\eqref{eqRea}, that is
$$\LRS\cC(\bY,\bSpec A)\;\xrightarrow{\sim}\; \Alg\cC(A,\Gamma(Y,\cO_Y)),$$ is now a standard exercise.
\end{proof}

\begin{example}\label{ExInf}
Consider $\cC=\Rep G$ for an infinitesimal group scheme $G$. It follows from \cite[Proposition~5.2.4]{ComAlg} and the above construction that, for $A\in\Alg\cC$, the locally ringed space $\bSpec A$ can be identified with the affine $k$-scheme $(X,\cO)$ corresponding to $A$, with appropriate structures of $G$-representations on the $k$-algebras $\cO(U)$.
\end{example}

Assume for the rest of the section that Hypothesis \ref{hypo} is satisfied. Recall the notation from Lemma~\ref{LemX0}.

\begin{remark}\label{SpecAf}\label{RemSpecA0}
\begin{enumerate}
\item It follows from the construction that, for $f\in A^0$ and $(X,\cO)=\bSpec A$, we have $\bSpec A_f\simeq (X_{f}, \cO|_{X_{f}})$.
\item Assume that $\rho_A$ is actually a homeomorphism (for instance $\cC$ is {\bf MN} and {\bf GR}). It then follows that for all $A\in\Alg\cC$, we have a canonical isomorphism
$$(\bSpec A)^0\;\xrightarrow{\sim}\; \bSpec (A^0).$$

\item If $\cC$ is {\bf MN} and {\bf GR}, it follows quickly from \cite[Theorem~5.3.1(1)]{ComAlg} that we have a canonical isomorphism
$$\bSpec (A_0)\;\xrightarrow{\sim}\; (\bSpec A)_0.$$
\end{enumerate}
\end{remark}

\begin{lemma}\label{cocomplete}
The category $\LRS\cC$ has all colimits and $\LRS\cC\to\Top$ is cocontinuous.
\end{lemma}
\begin{proof}
It suffices to consider coproducts and coequalisers. That $\LRS\cC$ has all coproducts is straightforward.

It is easy to see that coequalisers in the category of all (not necessarily locally) ringed spaces exist and the underlying topological spaces are also coequalisers. Concretely, for a parallel set of arrows $\bX\rightrightarrows \bY$ in $\LRS\cC$, we can define a ringed space on the equaliser~$Z$ in $\Top$ of $X\rightrightarrows Y$. The ring of functions is defined via the equalisers
$$\cO_Z(U)\;\to\; \cO_Y(V)\;\rightrightarrows \;  \cO_X(W)$$
for an open $U\subset Z$, with preimage $V\subset Y$ and primage $W\subset X$ under either composition.

We will also use the classical fact that for $\cC=\Vecc$, this coequaliser remains in the category of {\em locally} ringed spaces. For two arrows $\bX\rightrightarrows\bY$, we thus find a ringed space $\bY\to\bZ$ such that $\bY^0\to \bZ^0$ is a morphism in the category of ordinary locally ringed spaces. That $\bY\to\bZ$ is a morphism in $\LRS\cC$ follows from a direct application of \cite[Lemma~5.4.5]{ComAlg}, since $Y\to Z$ is surjective.
\end{proof}

\begin{remark}
That  $\Gamma:\LRS\cC\to\Aff \cC$ is cocontinuous follows from Theorem~\ref{ThmExistII} or from construction.
\end{remark}

\begin{prop}\label{SpecProd}\label{CorQuoSp}
Let $\bC$ be {\bf MN} and {\bf GR} (in particular \ref{hypo} is satisfied).
\begin{enumerate}
\item For morphisms $A\leftarrow C\to B$ in $\Alg\cC$, the canonical map
$$\Spec (A\otimes_CB)\to (\Spec A)\times_{\Spec C}(\Spec B)$$
is surjective. 
\item For a faithfully flat algebra morphism $A\to B$, the induced $\Spec B\to\Spec A$ makes $\Spec A$ into a quotient space of $\Spec B$.
\item For a faithfully flat algebra morphism $A\to B$, the diagram
$$\bSpec(B\otimes_AB)\rightrightarrows \bSpec B\to \bSpec A$$
is a coequaliser in $\LRS\cC$. 
\end{enumerate}
\end{prop}
\begin{proof}
For part (1), consider a pair $(\fp_1,\fp_2)$ in the target of the map, denote by $\fp$ the corresponding prime ideal in $C$. By \cite[Proposition~3.2.2]{ComAlg}, the domains $A/\fp_1$, $B/\fp_2$ and $C/\fp$ are $k$-algebras and we denote their fields of fractions by $F_1,F_2,F$. The preimage of any prime ideal in the (non-zero) target of
$$A\otimes_CB\to F_1\otimes_FF_2,$$
is sent to $(\fp_1,\fp_2)$.

That the map in part (2) is surjective follows from \cite[Theorem~4.2.6 and Proposition~4.3.1]{ComAlg}. That this makes $\Spec A$ into a quotient of $\Spec B$ can be proved precisely as in the case $\cC=\Vecc_k$, see for instance the proof of \cite[I.2.2.7]{DG}. In particular, part (1) of the proposition is used as the generalisation of the corresponding \cite[I.1.5.4]{DG}.

For part (3) we observe first that the diagram in $\mathsf{Top}$
$$\Spec(B\otimes_AB)\rightrightarrows \Spec B\to \Spec A$$
is a coequaliser. Indeed, since projections onto quotients are effective epimorphisms in $\Top$, this follows from the combination of parts (1) and (2). That the diagram in (3) is a coequaliser in the category of ringed spaces (so in particular in $\LRS\cC$) then follows from construction of the coequaliser in the proof of \autoref{cocomplete} and the observation that $A_f\to B_f$ is faithfully flat for each $f\in A^0$, so that $A_f$ corresponds to the equaliser of $B_f\rightrightarrows (B\otimes_AB)_f$, see \cite[Remark~2.3.4]{ComAlg}.
\end{proof}

\begin{remark}\label{Remz}
Consider the morphism $A\to \prod_i A_{f_i}$ for a finite set $\{f_i\in A^0\}$ which generates $A$. Then the conclusions in Proposition~\ref{SpecProd}(2) and (3) remain true without the additional assumptions {\bf MN} and {\bf GR}, with simplified proof.
\end{remark}

\subsection{Geometric realisations of functors}

We continue to assume Hypothesis~\ref{hypo}.

\subsubsection{} Using the functor $\bSpec$ from Theorem~\ref{ThmExistII}, we define the functor
\begin{equation}\label{eqFOP}
h:\LRS\cC\;\to\; \Fun\cC,\quad \bX\mapsto h_{\bX}=\LRS\cC(\bSpec -, \bX).
\end{equation}
Clearly, we have $h_{\bSpec A}\simeq \Phi_A$, but more generally:
\begin{example}\label{ExFunP}
Consider an algebra $A$ with ideal $I<A$, set $\bSpec A= (X,\cO)$ and let $U\subset X$ be the open complement of $V(I)$.
It is a standard exercise to verify that the functor in \eqref{eqFOP} sends $(U,\cO|_U)$ to $\Phi_A^I$.
\end{example}

\subsubsection{} If we take the convention that algebras are restricted in size by some cardinality, the category $\Alg\cC$ is essentially small. Since $\LRS\cC$ is cocomplete, by Lemma~\ref{cocomplete}, 
we have the left Kan extension of $\bSpec$ along the Yoneda embedding, which we denote by
$$|\cdot|:\;\Fun \cC\to\LRS\cC.$$
We refer to $|F|$ as the {\bf geometric realisation} of $F\in\Fun\cC$. To describe this explicitly we replace $\Alg\cC$ with a small equivalent category. Let $D(F)$ be the category of elements of $F$, in particular, the set of objects is $\sqcup_AF(A)$. We denote the forgetful functor $D(F)\to\Alg\cC$ by $a\mapsto A_a$.  We then have
$$|F|\;=\;\colim_{a\in D(F)^{\op}}\bSpec A_a.$$
As for instance in \cite[I.1.4]{DG}, one verifies directly that $|\cdot|$ is the left adjoint of $h$. More concretely, we have isomorphisms
$$\LRS\cC(|F|,\bX)\;\xrightarrow{\sim}\; \Fun(F,h_{\bX}),$$
where $|F|\to\bX$ is sent to the natural transformation which maps $a\in F(A)$ to the composite
$$\bSpec A\;\to\; |F|\;\to\; \bX,$$
where the left morphism is one of defining morphisms of the colimit $|F|$. In particular, we point out that $|\Phi_A|=\bSpec A$.

We summarise the functors in the following diagram
$$
\xymatrix{
\Aff\bC\ar[rrrr]^\Phi\ar@/^1pc/[ddrr]^{\bSpec}&&&&\Fun\bC\ar@/_1pc/[ddll]_{|\cdot|}\\
&&&&      &&|\cdot|\dashv h,\quad \Gamma\dashv\bSpec,\\
&&\LRS\cC\ar@/^1pc/[uull]^{\Gamma} \ar@/_1pc/[uurr]_{h}   &&&& |\cdot|\circ \Phi\simeq\bSpec,\;\;h\circ\bSpec \simeq\Phi.
}$$

\begin{lemma}\label{LemOpSubF}
Consider $F\in \Fun\cC$ and let $(X,\cO):=|F|$ be its geometric realisation in $\LRS\cC$.
\begin{enumerate}
\item A subset $U\subset X$ is open if and only if for each $A\in\Alg\cC$ and $f:\Spec A\to X$, induced from an element of $F(A)$ via
$$F(A)\simeq \Nat(\Phi_A,F)\xrightarrow{|\cdot|}\LRS\cC(\bSpec A,|F|)\to\Top(\Spec A,X),$$
the subset $f^{-1}(U)$ is open in $\Spec A$.
\item There is a canonical bijection between open subsets of $X$ and open subfunctors of $F$.
\end{enumerate}
\end{lemma}
\begin{proof}

By Lemma~\ref{cocomplete}, $X$ is the colimit of 
$$D(F)^{\op}\;\to\;\Aff\cC\;\xrightarrow{\Spec}\;\Top.$$
In particular a subset $U\subset X$ is open if and only if every inverse image under a defining $\Spec A\to X$ is open. This can be reinterpreted as part (1) of the lemma.

Now we prove part (2).
For an open subset $U\subset X$, denote, for each $a\in D(F)$, the inverse image in $\Spec A_a$ by $U_a$ and the corresponding radical ideal by $I_a<A_a$. Each morphism in $D(F)$ induces a commutative square
$$\xymatrix{
\Phi_{A_a}\ar[rr]&& \Phi_{A_b}\\
\Phi_{A_a}^{I_a}\ar[rr]\ar@{^{(}->}[u]&& \Phi^{I_b}_{A_b}\ar@{^{(}->}[u]
}$$
which is actually a fibre product. It follows that the colimit over $\{\Phi_{A_a}^{I_a}\,|\, a\in \Ob D\}$ yields an open subfunctor of $F$.

On the other hand, consider a subfunctor $G\subset F$, such that each $\Phi_{A_a}\times_FG$ is an open subfunctor of $\Phi_{A_a}$. These open subfunctors define open subsets $U_a\subset \Spec A_a$. Similarly to the above these can be used to define an open subset of $X$, which yields an inverse to the previous assignment.
\end{proof}

\begin{lemma}\label{LemFais}
Assume that $\cC$ is {\bf MN} and {\bf GR}.
\begin{enumerate}
\item The functor $h:\LRS\cC\to\Fun\cC$ takes values in $\Fais\cC$.
\item For $F\in \Fun\cC$ and $F\to F_1\in\Fais\cC$ its sheafification, $|F|\to |F_1|$ is an isomorphism.
\end{enumerate}
\end{lemma}
\begin{proof}
We prove part (1). That $h_{\bX}$ commutes with products follows easily. By \autoref{SpecProd}(3), for a faithfully flat algebra morphism $A\to B$, the diagram
$$h_{\bX}(A)\to h_{\bX}(B)\rightrightarrows h_{\bX}(B\otimes_AB)$$
is an equaliser.

Part (2) follows from part (1), the characterisation as sheafification as the left adjoint of $\Fais\cC\hookrightarrow\Fun\cC$ and the adjunction $|\cdot|\dashv h$.
\end{proof}

Using Remark~\ref{Remz} yields the following analogous result.
\begin{lemma}\label{LemFaisz}
\begin{enumerate}
\item The functor $h:\LRS\cC\to\Fun\cC$ takes values in $\Fais_z\cC$.
\item For $F\in \Fun\cC$ and $F\to F_1\in\Fais_z\cC$ its sheafification, $|F|\to |F_1|$ is an isomorphism.
\end{enumerate}
\end{lemma}

\begin{remark}
Consider the diagram of functors
$$
\xymatrix{\Aff\cC\ar[rr]^{\Phi}\ar[d]_{A\mapsto A_0}&&\Fun\cC\ar[d]\ar[rr]^{|\cdot|}&&\LRS\cC\ar[rrd]\ar[d]_{\bX\mapsto \bX_0}\\
\Aff_k\ar[rr]^{\Phi}&&\Fun_k\ar[rr]^{|\cdot|}&&\LRS_k\ar[rr]&&\Top,}$$
where the functors to $\mathsf{Top}$ just take the underlying topological space and the functor from $\Fun\cC$ to $\Fun_k$ restricts a functor to $\Alg_k\subset\Alg\cC$. The left square is always commutative. The rectangle of functors from $\Aff\cC$ to $\LRS_k$ is commutative when $\cC$ is {\bf MN} and {\bf GR}, see Remark~\ref{SpecAf}(3). Under those assumptions, it thus follows from Lemma~\ref{cocomplete} that that the topological space of $|F|$, $F\in\Fun\cC$ only depends on the image of $F$ in $\Fun_k$.
\end{remark}


\section{Schemes}\label{SecSch}
Let $\cC$ be a tensor category that satisfies Hypothesis~\ref{hypo}.

\subsection{Definitions}

\begin{definition}\label{DefScheme}
A {\bf $\bC$-scheme} is a locally ringed space $(X,\cO)\in \LRS\bC$ for which there exists an open cover $\{U\subset X\}$ such that for each $U$ there exists $A_U\in \Alg\bC$ with $(U,\cO|_U)\simeq \bSpec A_U$. 

The category of $\bC$-schemes $\Sch\bC$ is the corresponding full subcategory of $\LRS\bC$.
\end{definition}

\begin{remark}
The fully faithful functor $\bSpec:\Aff\cC\to\LRS\cC$ from Theorem~\ref{ThmExistII}, by definition, takes values in $\Sch\cC\subset\LRS\cC$. We will henceforth refer to an affine scheme either as an object in $\Aff\cC$ or the corresponding $\bSpec A$ in $\Sch\cC$.
\end{remark}

\begin{lemma}\label{LemXf}
For a $\bC$-scheme $\bX=(X,\cO)$ and $f\in A^0$, for $A=\Gamma(X,\cO)$, we have a canonical isomorphism
$$A_f=\cO(X)_f\;\xrightarrow{\sim}\; \cO(X_f).$$
\end{lemma}
\begin{proof}
For an affine scheme, this is by construction. For the general case, this follows by considering an affine cover and using exactness of localisation.
\end{proof}

\begin{lemma}\label{LemSober}
For a $\bC$-scheme $\bX=(X,\cO)$, the topological space is sober (taking the closure yields a bijection between points of $X$ and irreducible closed subspaces of $X$).
\end{lemma}
\begin{proof}
As in the classical case, this follows from the affine case, since being sober is a local property.
\end{proof}

\subsubsection{Fibre product of schemes}
By the definition of $\Spec$ as an adjoint, it follows that
$$\bSpec(A\otimes_CB)\;\simeq\;\bSpec A \times_{\bSpec C}\bSpec B$$
in $\Sch\cC$. More general fibre products in $\Sch\cC$ are defined as in the classical by gluing together the above fibre product.

\begin{definition}
\begin{enumerate}
\item A morphism of $\cC$-schemes $f:\bX\to\bY$ is {\bf surjective} (resp. {\bf injective}) if the underlying map $X\to Y$ is surjective (resp. injective).
\item A morphism of $\cC$-schemes $f:\bX\to\bY$ is {\bf flat} if $\cO_{Y,f(x)}\to \cO_{X,x}$ is a flat morphism for every $x\in X$. The morphism $f$ is {\bf faithfully flat} if it is flat and surjective.
\end{enumerate}
\end{definition}

\begin{remark}
It follows from the basic theory of \cite{ComAlg} (for instance Propositions~4.3.1 and 5.4.4 and Lemma~4.4.3), that an algebra morphism $A\to B$ is (faithfully) flat if and only if $\bSpec B\to \bSpec A$ is (faithfully) flat.
\end{remark}

\begin{lemma}\label{LemMono}
Assume that $\cC$ is {\bf MN} and {\bf GR}, and let $f:(X,\cO_X)=\bX\to\bY=(Y,\cO_Y)$ be a monomorphism in $\Sch\cC$.
\begin{enumerate}
\item $f$ is injective;
\item for every $x\in X$, the composite morphism $\cO_{Y,f(x)}\;\to\;\cO_{X,x}\;\tto\; \cO_{X,x}/\fm\cO_{X,x}$ is an epimorphism in $\Ind\cC$, with $\fm$ the maximal ideal in $\cO_{Y,f(x)}$.
\end{enumerate}
\end{lemma}
\begin{proof}
First we consider part (1). Since every simple algebra is a field, for every point $x\in X$ for a scheme $\bX=(X,\cO)$ there exists a morphism of schemes $\bSpec K\to \bX$, for a field extension $K/k$ (for instance the quotient of $\cO_x$ by its maximal ideal), which sends the unique point to $x$.
We can then use the proof of the corresponding statement in \cite[Lemme I.1.5.3]{DG}.

Part (2) reduces to the affine case, so $\bX=\bSpec A$, $\bY=\bSpec B$ and $f$ is $\bSpec$ of $\phi:B\to A$. Hence, $\phi$ is an epimorphism in $\Alg\cC$. The composite morphism in part (2) is then easily seen to be an epimorphism in $\Alg\cC$ as well. As it factors through the epimorphism (in $\Alg\cC$) $\cO_{Y,f(x)}/\fm\to \cO_{X,x}/\fm\cO_{X,x}$ and $\cO_{Y,f(x)}/\fm$ is a field, this latter epimorphism must be an isomorphism, concluding the proof.
\end{proof}

\subsection{Quasi-coherent sheaves}

\subsubsection{} Consider a morphism $\phi:\bX\to\bY$ of $\cC$-schemes. We use the same notation for the map of topological spaces $\phi:X\to Y$. For an $\bX$-module $\cF$, using $\cO_Y\to\phi_\ast\cO_X$, one can interpret~$\phi_\ast\cF$ as a $\bY$-module. Using part (1) of the following proposition, one verifies that~$\phi_\ast$ induces a functor
$$\phi_\ast:\;\QCoh_{\bC}\bX\to\QCoh_{\bC}\bY.$$

\begin{prop}\label{PropQCoh2}
Let $\bX$ be a $\cC$-scheme.
\begin{enumerate}
\item An $\bX$-module is quasi-coherent if and only if there exists an open affine cover such that for each $U\subset X$ in the cover and $f\in \cO_{\bX}(U)^0$, the restriction map yields an isomorphism $\cM(U)_f\to\cM(U_f)$.
\item If $\bX$ is affine with algebra of global sections $A$, then
$$\Gamma(X,-): \QCoh_{\bC}\bX\;\to\; \Mod_{\bC} A$$
is an equivalence.
\item $\QCoh_{\cC}\bX$ is an abelian subcategory of the category of all $\bX$-modules.
\end{enumerate}
\end{prop}
\begin{proof}
The same proof as for \autoref{LemQCoh1} applies to parts (1) and (2). Also part (3) is standard. Indeed, for the affine case, consider the inverse $M\mapsto \widetilde{M}$ of the equivalence in part~(2), where $\widetilde{M}(X_f)=M_f$. Since this functor is fully faithful and localisation exact, it follows that (co)kernels belong to $\QCoh\bX$. The general case follows then by the fact that exactness is defined locally.
\end{proof}

The notation $\cF/\cG$ for a quasi-coherent $\bX$-module $\cF$ with submodule $\cG$ on a $\cC$-scheme will thus refer to the cokernel in $\QCoh\cC$, or equivalently in the category of all $\bX$-modules, given by the sheafification of $U\mapsto \cF(U)/\cG(U)$.

\subsubsection{}\label{DefJV} Let $\bX=(X,\cO)$ be a $\bC$-scheme. Consider a fixed closed subset $V\subset X$. We define the subpresheaf $\cJ^V\subset\cO$ where, for each open $U\subset X$, $\cJ^V(U)$ is the intersection of the preimages of the maximal ideals for all $\cO(U)\to \cO_x$ with $x\in V\cap U$. In particular $\cJ^V(U)=\cO(U)$ if $U\cap V=\varnothing$.

\begin{prop}\label{PropSupp}
Let $\bX=(X,\cO)$ be a $\bC$-scheme.
\begin{enumerate}
\item For a quasi-coherent ideal sheaf $\cJ<\cO$, the subspace $\supp(\cO/\cJ)\subset X$ is closed.
\item For closed $V\subset X$, $\cJ^V$ is a quasi-coherent ideal sheaf in $\cO$, with $\supp (\cO/\cJ^V)=V$.
\item For closed $V\subset X$ and $f\in \Gamma(X,\cO)^0$, we have $X_f\cap V=\varnothing$ if and only if $f\in \Gamma(X,\cJ^V)$.
\end{enumerate}
\end{prop}
\begin{proof}
Part (1) can be reduced to the affine case where we observe that the quotient of the ideal sheaf associated to $I<A$ has support $V(AI^0)$.

Part (2) follows from a reduction to the affine case and exactness of localisation. Note that for $\bX=\bSpec A$, the ideal $\cJ^V(X)<A$ is the radical ideal corresponding to $V\subset \Spec A$.

Part (3) is immediate from the definitions of $\cJ^V$ and $X_f$, see~\ref{DefXf}.
\end{proof}

\begin{example}\label{ExNil}
If we take $V=X$ in part (2), then $\cJ^X(U)=\Nil\cO(U)$. Indeed, observe first that $\Nil\cO$ is a sheaf. The inclusion $\Nil\cO \subset \cJ^X$ is obvious. That the inclusion is an equality can be reduced to the affine case, where it is a consequence of \cite[Lemma~2.2.5]{ComAlg}.
\end{example}

\subsection{Subschemes}
Consider a $\bC$-scheme $\bX=(X,\cO)$.

\subsubsection{Open subschemes} It follows from Remark~\ref{SpecAf}(1) that for an open subspace $U\subset X$, the locally ringed space $\bU:=(U,\cO|_U)$ is again a scheme. These are the open subschemes of~$\bX$.


\subsubsection{Closed subschemes}
For a quasi-coherent ideal sheaf $\cJ<\cO$, let $Y=\supp(\cO/\cJ)$ be the closed subset of $X$ from Proposition~\ref{PropSupp}(1). 
Again by reducing to the affine case, it follows that $(Y,\cO/\cJ)$ is a $\bC$-scheme. These are the closed subschemes of $\bX$.

\begin{example}
\begin{enumerate}
\item Proposition~\ref{PropSupp}(2) allows us to associate a canonical closed subscheme to any closed subspace of $X$. In particular we define the closed subscheme
$$\bX_{\red}\subset \bX$$
corresponding to the ideal in Example~\ref{ExNil}.
\item For any $A\in\Alg\cC$ one verifies directly that $R(A_f)=R(A)_f$. It then follows that 
$$U\mapsto R(\cO(U))$$
is a quasi-coherent ideal sheaf (included in $\Nil\cO$ when $\cC$ is {\bf MN}). The corresponding closed subscheme is denoted by $\bX_0$, which is consistent with the notation of Lemma~\ref{LemX0}(2). Assuming $\cC$ is {\bf MN}, we have
$$\bX_{\red}\;\subset\; \bX_0\;\subset\;\bX.$$
\end{enumerate}
\end{example}

%

\begin{lemma}\label{LemAffNil}
Consider a nilpotent quasi-coherent ideal sheaf $\cJ<\cO$. If the closed subscheme of $\bX$ associated to $\cJ$ is affine, then so is $\bX$.
\end{lemma}
\begin{proof}
The proof of \cite[I.\S 2.6.1]{DG} carries over verbatim.
\end{proof}

\begin{definition}An {\bf open} (resp. {\bf closed}) {\bf immersion} of $\bC$-schemes is a morphism $\bY\to\bX$ which restricts to an isomorphism between $\bY$ and an open (resp. closed) subscheme of $\bX$.
\end{definition}

\begin{example}
For $x\in \bX(\unit)$, the corresponding $\Spec \unit\to\bX$ is a closed immersion.
\end{example}

\begin{definition}
 An immersion is a morphism $\bZ\to\bX$ of $\cC$-schemes which is a composition
$$\bZ\xrightarrow{\mbox{closed immersion}}\bY\xrightarrow{\mbox{open immersion}}\bX.$$
\end{definition}

\begin{lemma}\label{LemImmUnion}
Consider a morphism $f:\bX\to \bY$ of $\cC$-schemes and a collection of open subsets $\{B_i\subset Y\mid i\}$, for which we set $U_i:=f^{-1}(B_i)\subset X$. If $X=\cup_iU_i$ and each $\bX|_{U_i}\to \bY|_{B_i}$ is an immersion, then $f$ is an immersion.
\end{lemma}
\begin{proof}
By assumption, there exist open $C_i\subset B_i$ such that $f(U_i)\subset C_i$ and $\bX|_{U_i}\to \bY|_{C_i}$ is a closed immersion. We can thus replace $B_i$ by $C_i$ and henceforth only consider closed immersions. We can subsequently also replace $Y$ by $\cup_i B_i$. 

Now it remains to be proved that for an open cover $\{B_i\subset Y\}$ for which each $\bX|_{U_i}\to \bY|_{B_i}$ is a closed immersion, it follows that $f$ is a closed immersion. This follows easily, see also \cite[I.2.4.9]{DG}. Indeed, we can quickly observe that the kernel of $\cO_Y\to f_\ast\cO_X$ is quasi-coherent and that the morphism is an epimorphism.
\end{proof}

\subsection{Quasi-affine schemes}

As usual, we call a scheme {\bf quasi-compact} if its underlying space is quasi-compact.

 \begin{definition}
 A scheme $\bX$ is {\bf quasi-affine} if it is isomorphic to an open quasi-compact subscheme of an affine scheme.
\end{definition}

\begin{prop}\label{PropQA}
For a quasi-compact scheme $\bX=(X,\cO)$, set $A=\Gamma(X,\cO)$. The following are equivalent:
\begin{enumerate}
\item $\bX$ is quasi-affine.
\item The canonical morphism
$\bX\to\bSpec A$
is an open immersion.
\item Taking global sections yields a faithful functor
$$\Gamma(X,-):\;\QCoh_{\cC}\bX\to \Mod_{\cC}A.$$
\item There exists an algebra $R\in\Alg\cC$ and an immersion $\phi:\bX\to\bSpec R$.
\item The set $\{X_f,f\in A^0\}$ of open subsets of $X$, see~\ref{DefsLRS}, is a basis for $X$.
\end{enumerate}
\end{prop}
\begin{proof}
By definition, (2) implies (1) and (1) implies (4). To prove that (4) implies (3), we observe that, as for any immersion, $\phi_\ast$ is faithful. By standard properties of adjunction, $\phi$ factors through $\bSpec A$. Using the equivalence in Proposition~\ref{PropQCoh2}(2) then yields a commutative square
$$\xymatrix{
\QCoh_{\cC} \bX\ar[rr]^{\Gamma}\ar[d]_{\phi_\ast}&& \Mod_{\cC} A\ar[d]_{\Res}\\
\QCoh_{\cC}\bSpec R\ar[rr]^{\sim}&&\Mod_{\cC} R.
}$$
It follows that $\Gamma$ is faithful.

Now we prove that (3) implies (5). By equation~\eqref{GlobSec}, applied to the case where $N$ is a generator $G$ of $\Ind\cC$, it follows from faithfulness of $\Gamma(X,-)$ that $\cO\otimes G$ is a generator of $\QCoh_{\cC}\bX$. Consequently, given $\cM$ in $\QCoh_{\cC}\bX$, there exists some $N\in\Ind\cC$ with an epimorphism $\cO\otimes N\tto\cM$. As this epimorphism, again by \eqref{GlobSec}, factors via $\cO\otimes \Gamma(X,\cM)$ we find that for
every $x\in X$, the canonical morphism
$$\cO_{x}\otimes \Gamma(X,\cM)\;\to\; \cM_x$$
is an epimorphism.

Take an arbitrary $x\in X$ and open neighbourhood $U\ni x$ with complement $V=X\backslash U$. Consider the quasi-coherent ideal sheaf $\cJ^V<\cO$ from \ref{DefJV}. By the above paragraph and $\cJ^V_x=\cO_{x}$, we find
$$\cO_{x}\otimes \Gamma(X,\cJ^V)\;\tto\; \cO_{x}.$$
In other words, the ideal $J:=\Gamma(X,\cJ^V)$ in $A$ is not contained in the inverse image $\fm<A$ of the maximal ideal in $\cO_{x}$. By \ref{hypo}(2), this implies that $J^0$ is not contained in $\fm$. In other words, there exists
 $f\in \Gamma(X,\cJ^V)^0$ which is not mapped to the maximal ideal of $\cO_{x}$. By definition $x\in X_f$, and by Proposition~\ref{PropSupp}(3), $X_f\subset U$, so (5) follows.

Now we prove that (5) implies (2). Consider the set $S\subset A^0$ of those $f$ for which $\bX_f$ is affine. Then also $\{X_f,f\in S\}$ forms a basis for $X$. Indeed, for every $x\in X$ and affine open $U\ni x$, there exists $x\in X_f\subset U$. But then $X_f=U_f$, which is also affine.
Lemma~\ref{LemXf} shows that, for $f\in S$, $\bX_f$ is canonically isomorphic to $\bSpec A_f$, which implies
 that the morphism of affine schemes $\bX_f\to\bX\to\Spec A$ is an open immersion. This implies (2).
 \end{proof}


\begin{remark}
Assume that $\rho_A$ is a homeomorphism for all $A\in\Alg\cC$. It follows from \autoref{RemSpecA0}(2) that $\bX^0$ is a $k$-scheme, for a $\cC$-scheme $\bX$. Moreover, the condition in \autoref{PropQA}(5) clearly being equivalent for $\bX$ and $\bX^0$ then implies that $\bX$ is quasi-affine if and only if $\bX^0$ is quasi-affine.
\end{remark}

\subsection{Functor of points}
We consider the restriction to $\Sch\cC$ of functor $h$ from~\eqref{eqFOP}:
\begin{equation}
\label{eqh}
\Sch\cC\;\to\;\Fun\cC,\quad \bX\mapsto h_{\bX}=\Sch\cC(\bSpec-,\bX).
\end{equation}

\begin{theorem}\label{CompThm}
\begin{enumerate}
\item The functor \eqref{eqh} is fully faithful. Consequently, the canonical morphism $|h_{\bX}|\to\bX$ is an isomorphism for every $\bX\in\Sch\cC$.
\item The functor \eqref{eqh} yields an equivalence between $\Sch\cC$ and the full subcategory of $F\in \Fais_z\cC$ which admit an open cover by representable functors as in \ref{DefOC}.
\item If $\cC$ is {\bf MN} and {\bf GR}, one can replace $\Fais_z\cC$ by $\Fais\cC$ in part (2).
\end{enumerate}
\end{theorem}
\begin{proof}
Fully faithfulness follows precisely as in \cite[Proposition VI-2]{EH}.

Now we prove part (2). That $h_{\bX}$ is in $\Fais_z\cC$ is observed in Lemma~\ref{LemFaisz}(1).
Furthermore, an open affine cover of $\bX$ yields an open cover of $h_{\bX}$. Indeed, consider an open affine $\bSpec A\subset X$. For any element $\bSpec B\to \bX$ of $h_{\bX}(B)$, we have a radical ideal $I<B$ associated to the inverse image in $\Spec B$ of $\Spec A$. It follows easily that $\Phi_A\times_{h_{\bX}}\Phi_B=\Phi_B^I$, showing that $\Phi_A\to h_{\bX}$ is an open subfunctor. The analysis in~\ref{DefOC} shows that in this way an open (affine) cover of a scheme leads to an open (representable) cover of the associated functor.

To prove the other direction, consider $F\in \Fun\cC$ with an open cover $\{\Phi_{R_i}\hookrightarrow F\}$. Applying the definition of open subfunctors to each pair $\Phi_{R_i}\to F\leftarrow \Phi_{R_j}$ yields ideals $I_{ij}<R_i$ for which
$$\Phi^{I_{ij}}_{R_i}\;\simeq\; \Phi^{I_{ji}}_{R_j}.$$
Applying part (1) and Example~\ref{ExFunP} yields isomorphisms
$$\bSpec R_i|_{U_{ij}}\;\simeq\;\bSpec R_j|_{U_{ji}} $$
where $U_{ij}\subset \Spec R_i$ is the complement of $V(I_{ij})$. This allows us to glue together the affine schemes $\bSpec R_i$ into a scheme. By construction, we have a natural transformation $F\to h_{\bX}$, which is an isomorphism when $F$ is a faisceau in the Zariski topology.

Part (3) follows from Lemma~\ref{LemFais}(1).
\end{proof}

\begin{remark}\label{RemTV2}
By Proposition~\ref{PropTV}, Remark~\ref{RemOpen}(2) and Proposition~\ref{PropCover}(2), it follows that the full subcategory of $\Fais_z\cC$ corresponding to $\cC$-schemes in Theorem~\ref{CompThm}(2) is exactly the category of `schemes relative to $\Ind\cC$' in \cite[D\'efinition~2.15]{TV}.

Equivalently, we can interpret Theorem~\ref{CompThm} as saying that the functor of points provides an equivalence between the category of $\cC$-schemes (defined geometrically in Definition~\ref{DefScheme}) and the category of `schemes relative to $\Ind\cC$' (defined categorically in the general theory of \cite{TV}). Hence, we have constructed a geometric realisation of the notion of schemes from \cite{TV} applied to suitable tensor categories.
\end{remark}

\begin{remark}
Based on Theorem~\ref{CompThm}, we say that $F\in \Fun\cC$ `is a scheme', if it is of the form $h_{\bX}$ for a scheme $\bX$. In other words,  $F\in \Fun\cC$  `is a scheme' if $|F|$ is a scheme such that $F\to h_{|F|}$ is an isomorphism.
\end{remark}

%
%
%


\section{Algebraic schemes}\label{SecAlgSch}
Assume now that $\bC$ is {\bf MN} and {\bf GR}.
We demonstrate that the notation of algebraic schemes and their basic properties, as derived for instance in \cite[\S I.3]{DG}, carry over without issues.

\subsection{Definitions}

\begin{definition}
\begin{enumerate}
\item A $\bC$-scheme $\bX=(X,\cO)$ is {\bf algebraic} if it has a finite open affine cover $\{U_\alpha\}$ such that $\cO(U_\alpha)$ is finitely generated for each $\alpha$.
\item A $\bC$-scheme $\bX=(X,\cO)$ is {\bf locally algebraic} if it has an open affine cover $\{U_\alpha\}$ such that $\cO(U_\alpha)$ is finitely generated for each $\alpha$.
\end{enumerate}
\end{definition}

We present some properties of algebraic schemes, with proofs which are direct analogues of the classical ones, found for instance in \cite{DG}.

\begin{remark}
\begin{enumerate}
\item For an algebraic $\bC$-scheme $\bX=(X,\cO)$, the topological space $X$ is noetherian. 
\item A locally algebraic $\bC$-scheme $\bX=(X,\cO)$ is algebraic if and only if $X$ is quasi-compact.
\item If $\bX$ is locally algebraic, then for any affine open $U\subset X$, the algebra $\cO(U)$ is finitely generated. In particular $\bSpec A$ is algebraic if and only if $A$ is finitely generated.
\end{enumerate}
\end{remark}

\begin{lemma}\label{LemLocFin}
If $\bX=(X,\cO)$ is a locally algebraic $\bC$-scheme, then there exists $x\in X$ for which the maximal ideal in $\cO_x$ is nilpotent.
\end{lemma}
\begin{proof}
The claim reduces immediately to the case $\bX=\Spec A$ for a finitely generated (and hence noetherian) algebra $A$. Let $\fp$ be a minimal prime ideal (which exists by Zorn's lemma). Then, by construction $\cO_{\fp}$ is noetherian (as a localisation of a noetherian algebra) and has only one prime ideal, the maximal ideal. By \cite[Lemma~2.2.5]{ComAlg} the maximal ideal is nil, so also nilpotent as it is finitely generated.
\end{proof}

\subsection{Immersions}

The following lemma is straightforward but useful.
\begin{lemma}\label{LemTop}
For a point in a topological space $t\in T$, we write $\overline{t}$ for the closure in $T$ of the corresponding singleton. For a map $T_1\to T_2$ in $\mathsf{Top}$ and $t\in T_1$, we have
\begin{enumerate}
\item $\overline{t}\subset f^{-1}\left(\overline{f(t)}\right)$;
\item $u\in \overline{t}\;\Rightarrow\; f(u)\in \overline{f(t)}$.
\end{enumerate}
\end{lemma}

\begin{corollary}\label{CorUB}
For an injective morphism $f:\bX\to \bY$ in $\Sch\cC$ with $\bX$ algebraic, and a non-empty open $U\subset X$, there exists an open $B\subset Y$ for which $\varnothing \not=f^{-1}(B)\subset U$.
\end{corollary}
\begin{proof}
Since $X$ is noetherian, it is a finite union of irreducible components. Hence, by Lemma~\ref{LemSober}, $X$ is a finite union of closures of distinct points $x_1,\cdots, x_l \in X$. 
 Since $f$ is injective, by Lemma~\ref{LemSober} applied to $Y$, we can relabel the $x_i$ so that firstly $f(x_1)\not\in \overline{f(x_j)}$ for $j>1$ and secondly $x_1\in U$.

Similarly, since $X\backslash U$ is noetherian, it is a finite union of irreducible components. Hence, by Lemma~\ref{LemSober}, $X\backslash U$ is a finite union of closures (taken in $X$ or equivalently in $X\backslash U$) of points $u_1,\cdots, u_d \in X\backslash U$.

We set 
$$B':=\cup_i \overline{f(u_i)}\subset Y.$$
By Lemma~\ref{LemTop}(1), we have
$$X\backslash U\;  \subset\; f^{-1}(B')=B'\cap X.$$
We claim that furthermore
$$x_1\not\in f^{-1}(B'),$$
so that that we can take $B$ to be the complement of $B'$, proving the corollary.

We prove the remaining claim by contradiction. Assume that $f(x_1)\in \overline{f(u_i)}$ for some~$i$. Since $u_i\in \overline{x_j}$ for some $j$, Lemma~\ref{LemTop}(2) shows $f(u_i)\in \overline{f(x_j)}$.
In conclusion, we find $f(x_1)\in \overline{f(x_j)}$, which is indeed a contradiction, by the first paragraph, unless $j=1$. Lemma~\ref{LemSober} implies that the latter option implies $u_i=x_1$, which contradicts $x_1\in U$.
\end{proof}

\begin{lemma}\label{LemImm100}
Let $f:(X,\cO_X)=\bX\to \bY=(Y,\cO_Y)$ be a morphism of $\bC$-schemes where $\bX$ is locally algebraic. If there is a point $x\in X$ for which $\cO_{\bY,f(x)}\to \cO_{X,x}$ is an epimorphism in $\Ind\cC$, there exists an open neighbourhood $x\in U\subset X$ for which $\bX|_U\to \bY$ is an immersion.
\end{lemma}
\begin{proof}
The statement reduces to the case $\bX=\bSpec A$ and $\bY=\bSpec B$, where $A$ is finitely generated (and hence noetherian). Denote the algebra morphism corresponding to $f$ by $\phi:B\to A$.

Since the kernel of $A\to \cO_{\bX,x}$ is finitely generated, it follows that $A_t\to \cO_{\bX,x}$ is a monomorphism for some $t\in A^0$. Since $A_t$ is finitely generated, there must be $s\in B$ for which the image of $B_s\to \cO_{X,x}=A_{\fp}$ contains $A_t$. It then follows easily that
$$B_s\to A_{t\phi(s)}$$
is an epimorphism in $\Ind\cC$. We can thus take $U=X_{t\phi(s)}$
\end{proof}

\begin{lemma}\label{DisgustingLemma}
For a monomorphism $f:\bX\to \bY$ in $\Sch\cC$, with $\bX$ algebraic, there exists an open $B\subset Y$ such that $f:\bX|_{f^{-1}(B)}\to\bY|_B$ is an immersion and $f^{-1}(B)\not=\varnothing$.
\end{lemma}
\begin{proof}
We claim that there exists a non-empty open $U\subset X$ for which $\bX|_{U}\to \bY$ is an immersion. The conclusion then follows from Corollary~\ref{CorUB}.

To prove the claim, we can apply the proof of \cite[Proposition~I.3.4.6]{DG}: By Lemma~\ref{LemImm100} it is sufficient to show that there exists $x\in X$ such that $\cO_{Y,f(x)}\to \cO_{X,x}$ is an epimorphism. If we take $x\in X$ as in Lemma~\ref{LemLocFin}, then it is sufficient to show that $\cO_{Y,f(x)}\to \cO_{X,x}/f(\fm)\cO_{X,x}$ is an epimorphism, for $\fm$ the maximal ideal in $\cO_{Y,f(x)}$, by Nakayma's lemma. Now that $\cO_{Y,f(x)}\to \cO_{X,x}/f(\fm)\cO_{X,x}$ is an epimorphism is true by Lemma~\ref{LemMono}(2).
\end{proof}


\subsection{Rational points}
Let $\bX=(X,\cO)$ be a $\cC$-scheme.
\subsubsection{}
Note that sending a morphism $\bSpec\unit\to \bX$ to the image of the underlying map $\ast\to X$ of topological spaces yields an injection. We thus view 
\begin{equation}\label{X1}\bX(\unit)\;\subset\; X\end{equation}
as a subspace (with subspace topology).

\begin{lemma}\label{LemXfin}
Assume that $k$ is algebraically closed and that $\bX$ is locally algebraic. Then
$U\mapsto U\cap\bX(\unit)$ yields a bijection between the open subsets of $X$ and $\bX(\unit)$.
\end{lemma}
\begin{proof}
We prove the corresponding statement for closed subsets, which is equivalent. We claim that an inverse of $V\mapsto V\cap \bX(\unit)$ is given by $W\mapsto \overline{W}$, for closed $W\subset \bX(\unit)$. The only non-trivial claim is that the inclusion
$$\overline{V\cap\bX(\unit)}\subset V$$
is an equality, for every closed $V\subset X$, under the given assumptions. For a contradiction, assume the equality would be strict, then there is an open $U\subset X$ (which we assume to be affine algebraic without loss of generality) for which $U\cap V\not=\varnothing$, but $U\cap V\cap \bX(\unit)=\varnothing$. Let $x\in U\cap V$ be a point closed in $U$. By the finite type assumption and \cite[Proposition~3.2.2]{ComAlg}, it follows that the corresponding simple quotient of $\Gamma(U,\cO)$ is $\unit$ and hence $x\in \bX(\unit)$, which contradicts $U\cap V\cap \bX(\unit)=\varnothing$.
\end{proof}

\section{Homogeneous spaces}
\label{SecGH}
In this section, we fix a {\bf MN} and {\bf GR} tensor category $\cC$, and we make the following hypothesis, which is for instance satisfied for $\cC=\sVec$, by \cite{MZ}.

We refer to \cite[\S 7]{ComAlg} for the basics about affine group schemes in tensor categories.

\subsubsection{Hypothesis}\label{hypoQuo} For every algebraic group $G$ in $\cC$, and every subgroup $H<G$, the quotient in $\Fun\cC$
$$G/_0H\;:\; A\mapsto G(A)/H(A)$$
is such that its sheafification in $\Fais\cC$ is an algebraic scheme (here, and below, we freely use Theorem~\ref{CompThm}). We denote this sheafification by $G/H$ and by assumption we thus have a co-equaliser
\begin{equation}\label{EqCoeq}
G\times H\;\rightrightarrows\; G\;\to\; G/H
\end{equation}
in $\Sch\cC$.

\subsection{Some consequences} Fix an algebraic group $G$ in $\cC$ with subgroup $H<G$.

\subsubsection{}\label{DefQuo1} We have an isomorphism
$$G\times H\;\simeq\; G\times_{G/H} G $$
in $\Fais\cC$ which comes from $G\times H\to G\times G$, $(g,h)\mapsto (g,gh)$.

\begin{prop}\label{QuoFP}
The quotient morphism $G\to G/H$ is faithfully flat and affine (the inverse image of every affine open in $G/H$ is an affine open in $G$).
\end{prop}
\begin{proof}
By \ref{DefQuo1}, the (left) projection $G\times_{G/H}G\to G$ is faithfully flat, as it is just the projection $G\times H\to G$. 

That the epimorphism $G\to G/H$ (in $\Fais\cC$) is then also faithfully flat is then a standard consequence, see for instance \cite[III.1.2.11]{DG}.
That the morphism is affine also follows exactly as in the classical case, see \cite[I.2.3.9]{DG}.
\end{proof}

To fix notation, for the coequaliser in \eqref{EqCoeq}, we write $p:G\to G/H$ and $q: G\times H\to G/H$ for either composite. We also write $G/H=(X,\cO)$.

\subsubsection{} For $M\in \Rep^\infty_{\cC}H$, we consider the equaliser
$$\cL(M):=\mathrm{Eq}(p_\ast\cO_G \otimes M\rightrightarrows q_\ast \cO_{G\times H}\otimes M)$$
as $\cC$-(pre-)sheaves on $X$. One arrow comes from $p_\ast\cO_G\to q_\ast \cO_{G\times H}$ corresponding to multiplication $G\times H\to G$. The other arrow comes from the co-action $M\to \cO(H)\otimes M$.

As in \cite[Example~5.10]{Jantzen} we find
$\cL(\unit)\simeq \cO.$ This gives $\cL(M)$ the structure of an $\cO$-module, which is easily seen to be a quasi-coherent sheaf.

\begin{theorem}\label{ThmL}
\begin{enumerate}
\item The functor
$$\cL:\Rep^\infty_{\cC}H\;\to\; \QCoh_{\cC}(G/H)$$
is faithful and exact.
\item We have a natural isomorphism $\Gamma(\cL(M))\simeq \Ind^G_HM$ in $\Ind\cC$, for each $M\in \Rep^\infty_{\cC}H$.
\end{enumerate}
\end{theorem}
\begin{proof}
Part (2) follows immediately from the definitions.

To prove that $\cL$ is faithful, it suffices to prove that the functor
$$\Rep^\infty_{\cC}H\to \Mod_{\cC}\cO_{G/H}(U),\quad M\to \cL(M)(U)$$ is a faithful functor for some affine open $U\subset X$. By Proposition~\ref{QuoFP}, $p:G\to G/H$ is affine, $p^{-1}(U)$ is then affine and, exactly as in the proof of \cite[Proposition~I.5.9(a)]{Jantzen}, we can prove
$$\cO_G(p^{-1}U)\otimes_{\cO_{G/H}(U)}\cL(M)(U)\;\simeq\; \cO_G(p^{-1}U) \otimes M$$
which shows faithfulness.

By Proposition~\ref{QuoFP}, $\cO_{G/H}(U)\to \cO_G(p^{-1}(U))$ is faithfully flat, so the above also shows that $\cL$ is exact.
\end{proof}

\subsection{Actions on schemes}
We let $\bX=(X,\cO)\in\Sch\cC$.
\subsubsection{}

Consider an affine group scheme $G$ in $\cC$ with action $a:G\times \bX\to\bX$ (satisfying the usual relations). 
For $x\in X(\unit)$, we have the corresponding morphism $a_x:G\to \bX$.

Denote by $G_x<G$ the subgroup functor which sends $A\in \Alg\cC$ to the equaliser of
$$G(A)\rightrightarrows \bX(A),$$ 
where one maps comes from $a_x$ and the other from $G\to \ast\xrightarrow{x} \bX$. Note that $G_x$ is a subgroup since we can realise it as the fibre product $G\times_X\ast$.

\begin{example}\label{ExStab} Consider a monomorphism $\alpha:\unit\to V$ in $\cC$ for a $G$-representation on $V$.
\begin{enumerate}
\item The stabiliser subgroup $G_\alpha<G$, see \cite[7.4.2]{ComAlg}, is an example of the above construction, for $\bX=\mA_V:=\bSpec(\Sym V^\vee)$ and $x\in \mA_V(\unit)\simeq\Hom(V^\vee,\unit)$ corresponding to $\alpha:\unit\to V$.
\item Assume that $\mP_V$ is represented by a scheme $\bX$. The transporter subgroup $G_U<G$, see \cite[7.4.3]{ComAlg}, for $U=\im\alpha\subset V$, is an example of the above construction, for $x=U\in\mP_V(\unit)$, in the interpretation of Remark~\ref{RemPX}.
\end{enumerate}
\end{example}

\subsubsection{}
The action $a$ induces group homomorphisms
\begin{equation}\label{3G1}
G(\unit)\to \Aut(\bX),\quad G(\unit)\to\Aut(X)\quad\mbox{and}\quad G(\unit)\to\Aut(\bX(\unit)),
\end{equation}
with the second is obtained from the first, and the third is directly obtained from $a(\unit)$. It follows also immediately from the definitions that the action of $G(\unit)$ on $X$ restricts to the action on $\bX(\unit)\subset X$, see \eqref{X1}. In particular, the third actions takes values in the automorphisms of $\bX(\unit)$ as a topological space.


\begin{lemma}\label{LemXYGB}
Consider $\cC$-schemes $\bX,\bY$ equipped with $G$-actions and a $G$-equivariant morphism $f:\bX\to \bY$. Assume that the only open $U\subset X$ which are $G(\unit)$-invariant are $\varnothing$ and $X$. If there exists an open $B\subset Y$, with $\varnothing\not=f^{-1}(B)$ such that $\bX|_{f^{-1}(B)}\to \bY|_B$ is an immersion, then $f$ is an immersion.
\end{lemma}
\begin{proof}
Since $f$ is $G$-equivariant, it follows that $\bX|_{g(f^{-1}(B))}\to \bY|_{g(B)}$ is also an immersion, for every $g\in G(\unit)$. Since the union of $\{g(f^{-1}(B))\,|\, g\in G(\unit)\}$, by assumption, must be $X$, Lemma~\ref{LemImmUnion} implies that $f$ is an immersion.
\end{proof}

\begin{lemma}\label{LemImm} Assume that $k$ is algebraically closed and $G$ is algebraic.
The morphism~$a_x:G\to\bX$ factors through a morphism
\begin{equation}\label{GxX}G/G_x\to \bX,\end{equation} which is an immersion
\end{lemma}
\begin{proof}
That $a_x$ factors through the quotient follows from universality, and \eqref{GxX} is clearly $G$-equivariant.

We claim that there are no non-trivial $G(\unit)$-stable open subsets on the underlying space of $G/G_x$. 
Since $G\to G/G_x$ is surjective, see Proposition~\ref{QuoFP}, it suffices to prove this for the action of $G(\unit)$ on $G$ itself, rather than on $G/G_x$. The latter now follows from Lemma~\ref{LemXfin} and compatibility of the actions in \eqref{3G1}.

By Lemma~\ref{LemXYGB}, it thus suffices to prove that $G/G_x$ restricts to an immersion on some open $B\subset X$ which restricts to a nonempty subset on $G/G_x$.

Clearly $G/_0G_x\to\bX$ is a mononomorphism in $\Fun\cC$. By the usual form of the sheafification functor, it follows that $G/_1 G_x=G/G_x\to\bX$ is still a monomorphism (in $\Fais\cC$ or $\Fun\cC$). The conclusion thus follows from Lemma~\ref{DisgustingLemma}.
\end{proof}

\subsection{Observability}

We can now extend \cite[Theorem~A.1.1]{ComAlg} from $\Vecc$ to arbitrary {\bf MN} and {\bf GR} tensor categories $\cC$ satisfying Hypothesis~\ref{hypoQuo}. Indeed, the proof carries over verbatim, where we need to use \autoref{PropQA}, \autoref{LemImm} and \autoref{ThmL} in lieu of the corresponding classical results.
\begin{theorem}\label{ThmOb}
Let $k$ be algebraically closed, and consider an algebraic group $G$ in $\cC$ (which is {\bf MN}, {\bf GR} and satisfies Hypothesis~\ref{hypoQuo}) with subgroup $H<G$. 
\begin{enumerate}
\item The following conditions are equivalent on $H<G$.
\begin{enumerate}
\item Every representation of $H$ extends to one of $G$;
\item The geometric induction functor $\Ind^G_H$ is faithful;
\item The subgroup $H<G$ is the intersection of a family of stabiliser subgroups;
\item The inclusion $H\to G$ is a regular monomorphism in the category of affine group schemes in $\cC$.
\end{enumerate}
Furthermore, if $G$ is of finite type, then these properties are equivalent to
\begin{enumerate}
\item[(e)] The quotient $\cC$-scheme $G/H$ is quasi-affine.
\end{enumerate}
\item There exists a unique subgroup $H<H'<G$ for which $H\to H'$ is an epimorphism and $H'< G$ satisfies conditions (a)-(d) in part (1).
\end{enumerate}
\end{theorem}

Similarly, \cite[Theorem~A.2.1]{ComAlg} extends to our setting:

\begin{theorem}\label{ThmStrong}
Let $k$ be algebraically closed, and consider an algebraic group $G$ in $\cC$ with subgroup $H<G$. The following properties are equivalent:
\begin{enumerate}
\item The subgroup $H<G$ is exact;
\item The subgroup $H<G$ is exact and observable;
\item The subgroup $H_{\red}<G_{\red}$ is exact;
\item The quotient $\cC$-scheme $G/H$ is affine.
\end{enumerate}
\end{theorem}

\begin{remark}
Since all assumptions on $\cC$ are valid for $\cC=\sVec$, Theorems~\ref{ThmOb} and~\ref{ThmStrong} apply to supergroups. Theorem~\ref{ThmStrong} for supergroups was obtained in~\cite[Theorem~5.2]{Zu}, while Theorem~\ref{ThmOb} for supergroups appears to be new.
\end{remark}


\section{Examples}\label{SecEx}

In this section, we give some examples of geometric objects in $\cC_1:=\Ver_4^+$. In particular, we let $k$ be a field of characteristic 2. It is easy to verify that $\cC_1$ is {\bf GR} and {\bf MN}, see also \cite[Theorems~9.2.1 and~9.3.7]{CEO}.

\subsection{Notation}

\subsubsection{} 
We consider the Hopf algebra $k[D]/D^2$ where $D$ is primitive.
Recall from \cite{Ve} the tensor category $\cC_1=\Ver_4^+$ given by the monoidal category of finite dimensional $k[D]/D^2$-modules, with symmetric braiding given by
$$V\otimes W\;\xrightarrow{\sim}\; W\otimes V,\quad v\otimes w\mapsto w\otimes v+(Dw)\otimes (Dv).$$

We will denote by $\Omega:\cC_1\to\Vecc$ the (non-symmetric) monoidal forgetful functor and denote by $P\simeq k[D]/D^2$ the indecomposable projective object in $\cC_1$. We choose a basis $\{v_1,v_2\}$ of $\Omega P$ so that $v_2$ spans the socle and $Dv_1=v_2$.

\subsubsection{Algebras}\label{RelVer4}An object $A\in \Alg\cC_1$ corresponds to an ordinary (non-commutative) $k$-algebra $R:=\Omega(A)$, equipped with
$$D=D_R\in \End_k(R),$$ with $D^2=0$,
such that (by primitivity of $D$) 
\begin{equation}
\label{Leibniz} D(ab)\;=\; D(a)b+aD(b)
\end{equation}
and (by commutativity of $A$ in $\Ind\cC_1$)
$$ab+ba\;=\; (Da)(Db),$$
for all $a,b\in R$. In particular, we have the following useful relations
$$D(a^2)=0=(Da)^2.$$

\begin{example}\label{SymP}
The algebra $\Sym P$ in $\Alg\cC_1$ has (commutative) underlying $k$-algebra
$$k[x,y]/y^2$$
with action of $D$ determined by \eqref{Leibniz} and $Dx=y$.
\end{example}

\subsection{An affine projective space}

\subsubsection{} We define the algebra $T\in \Alg\cC_1$ with underlying $k$-algebra $k[t]/t^4$ and with action of $D$ determined by
$$D(t)\;=\; t^2.$$
We can present this algebra as $(\Sym P)/(y-x^2)$.

Denote by $\cL_u$ the $T$-module such that the $k[t]/t^4$-module $\Omega(\cL_u)$ is the regular module, but with action of $D$ given by
$$D(1)=t,\quad D(t^2)=t^3.$$
This is clearly an invertible (in fact involutive) $T$-module, and we consider the element of~$\mP_P(T)$ determined (using $P^\ast\simeq P$) by
$$T\otimes P\;\tto\; \cL_u,\quad 1\otimes v_1\mapsto 1.$$

This elements corresponds (by the Yoneda lemma) to a natural transformation $\Phi_T\to \mP_P$.

\begin{prop}\label{PropProjAff}
The above yields an isomorphism in $\Fais\cC$
$$\Phi_{T}\;\simeq\;\mP_P.$$
In other words, $\mP_P$ is the affine $\cC$-scheme $\bSpec T$.
\end{prop}

The remainder of this section is dedicated to the proof of the proposition.

\begin{lemma}\label{ProjLim}
For a directed system $A_\alpha$ in $\Alg\cC_1$, the map $\varinjlim_\alpha\mP_P(A_\alpha)\to\mP_P(\varinjlim_\alpha A_\alpha)$ is injective.
\end{lemma}
\begin{proof} Set $A=\varinjlim A_\alpha$.
We need to prove that if, for some $\alpha$, two elements of $\mP_P(A_\alpha)$ represented by
$A_\alpha\otimes P\tto \cL_1 $ and $A_\alpha\otimes P\tto \cL_2 $ become identical in $\mP_P(A)$ then they are identical in $\mP_P(A_\beta)$, for some $\beta$. Indeed, we have natural isomorphisms
\begin{eqnarray*}
\Hom_A(A\otimes_{A_\alpha}\cL_i,A\otimes_{A_\alpha}\cL_j)&\simeq&\Hom(\unit,A\otimes_{A_\alpha}\cL_i^\vee\otimes_{A_\alpha}\cL_j)\\
&\simeq &\varinjlim_\beta \Hom(\unit,A_\beta\otimes_{A_\alpha}\cL_i^\vee\otimes_{A_\alpha}\cL_j)\\
&\simeq &\varinjlim_\beta \Hom_{A_\beta}(A_\beta\otimes_{A_\alpha}\cL_i,A_\beta\otimes_{A_\alpha}\cL_j).
\end{eqnarray*}
Hence, an isomorphism between the $A\otimes_{A_\alpha}\cL_i$ yields an isomorphism between the $A_\beta\otimes_{A_\alpha}\cL_i$ for some $\beta$. Similarly, if this isomorphism leads to a commutative triangle with $A\otimes P$, the same is true for $A_\beta\otimes P$ for some $\beta$.
\end{proof}

\subsubsection{}\label{CalcVer4}
For $A\in\Alg\cC$, we observe that $a\in\Omega(A)$ has a left inverse if and only if the subobject~$X$ of $A$ spanned by $\{a,Da\}$ generates $A$ as a left $A$-module. Indeed, one direction of this claim is immediate. On the other hand, if $X$ generates $A$, then there exist $b,c\in \Omega(A)$ for which
$$ba+c(Da)=1.$$
Squaring this equation and using the relations in \ref{RelVer4} implies
$$(bab+(Da)(Db)(Dc))a=1.$$

\begin{lemma}\label{LemLocVer4}
If $A\in\Alg\cC_1$ is local, then so is the (non-commutative) $k$-algebra $\Omega(A)$.
\end{lemma}
\begin{proof}
Let $\fm<A$ be the maximal ideal. Since $A/\fm$ is a field extension of $k$, see \ref{ConsMNGR}, it follows that $\Omega(\fm)<\Omega(A)$ is also a maximal ideal. Consider $a\in \Omega(A)$ not included in $\Omega(\fm)$. Then the subobject of $A$ spanned by $\{a,Da\}$ generates $A$ as an ideal, so by \ref{CalcVer4} $a$ admits a left inverse. The same argument shows that it admits a right inverse. Hence $a$ is a unit in $\Omega(A)$, proving that $\Omega(A)$ is local.
\end{proof}

\begin{proof}[Proof of Proposition~\ref{PropProjAff}]
We have, using the relations in \ref{RelVer4},
$$\Phi_T(A)\;=\; \{a\in \Omega(A)\,|\, Da=a^2\}.$$
In this interpretation, $\Phi_T(A)\to \mP_P(A)$ sends $a\in \Phi_T(A)\subset \Omega(A)$ to 
$$A\otimes P\;\tto\; \cL_u^a(A):=A\otimes_T\cL_u.$$
Concretely, $\cL_u^a(A)$ is the (invertible) $A$-module with $\Omega\cL_u^a(A)\simeq\Omega A$ as an $\Omega A$-module, but with action of $D$ given by $D\underline{x}=\underline{xa+Dx}$. Here we write $\underline{x}$ for $x\in \Omega A$ interpreted in $\cL_u^a(A)$. The morphism from $A\otimes P$ sends $1\otimes v_1$ to $\underline{1}$. From this formulation, it is clear that $\Phi_T\to\mP_P$ is a monomorphism.

By Lemmata~\ref{LemTech} and~\ref{ProjLim}, it suffices to prove that that $\Phi_T\to\mP_P$ is an isomorphism when evaluated on local algebras. Thus we let $A\in\Alg\cC_1$ be a local algebra. 

Let $\cL$ be an invertible $A$-module. By Lemma~\ref{LemLocVer4}, we have $\Omega(\cL)\simeq\Omega(A)$.
Consequently, $\cL\simeq\cL_u^a(A)$ for some $a\in\Omega A$.
Take $A\otimes P\tto\cL=\cL_u^a(A)$ corresponding to an element of $\mP_P(A)$. If we let $\underline{x}$, for some $x\in\Omega( A)$, be the image of $1\otimes v_1$, then
$$1\otimes v_2\;\mapsto \; \underline{Dx+xa}=\underline{Dx+ax+(Da)(Dx)}.$$
Since $A\otimes P\tto\cL$ is an epimorphism, it follows from the computation in~\ref{CalcVer4} that $x$ is a unit. We then see that we can represent the same element of $\mP_P(A)$ by the image of $a+x^{-1}Dx(1+Da)$ under $\Omega(A)\supset\Phi_T(A)\to\mP_P(A)$.
\end{proof}

\subsection{A normal subgroup}

\subsubsection{} We denote by $\mG$ the functor 
$$\Alg\cC_1\to\Grp,\quad A\mapsto (\Omega A)^\times.$$
This is an affine group scheme in $\cC_1$ with
$$\Omega(\cO(\mG))\;=\; k[x,x^{-1},y]$$
and $Dx=y$. In particular, $D(x^{-1})=x^{-2}y$. We can also observe that
$$\Delta(x)=x\otimes x\quad\mbox{and}\quad \Delta(y)=y\otimes x+ x\otimes y.$$
We also point out that $\mG$ is {\bf not} abelian.

\subsubsection{} We define the subfunctor $N<\mG$
$$\Alg\cC_1\to\Grp,\quad A\mapsto ( A^0)^\times.$$
This is a normal subgroup, corresponding to 
$$\cO(\mG)\tto \cO(N)=k[t,t^{-1}],\quad x\mapsto t.$$
Then the quotient group $Q:=G/N$, see \cite[7.2.3]{ComAlg},
satisfies
\begin{equation}\label{ks2}
k[s]/s^2\simeq\cO(Q)=\cO(\mG)^N \;\hookrightarrow \cO(\mG),\quad s\mapsto x^{-1}y.
\end{equation}

\begin{corollary}\label{Corses}
\begin{enumerate}
\item We have short exact sequence, see \cite[7.2.5]{ComAlg}, in the category of affine group schemes in $\cC_1$:
$$1\to N\to\mG\to \alpha_2\to 1$$
with $\alpha_2$ the Frobenius kernel of ordinary multiplicative group $\mG_m$.
\item The sheafification of 
$$\Alg\cC_1\to\Set,\quad A\mapsto (\Omega A)^\times/ (A^0)^\times$$
in $\Fais\cC_1$ is given by the affine ($k$-)scheme corresponding to $k[s]/s^2$. Moreover, the natural map $(\Omega A)^\times/ (A^0)^\times\to\Phi_{k[s]/s^2}(A) $ sends $a (A^0)^\times $ to the algebra morphism $s\mapsto a^{-1}Da$.
\end{enumerate}
\end{corollary}
\begin{proof}
Part (1) follows from the above combined with the observation that $x^{-1}y$ is primitive in $\cO(\mG)$.

For part (2) is an application of part (1), by \cite[Proposition~7.2.4(1)]{ComAlg} and \eqref{ks2}.
\end{proof}
\subsection{A (non-normal) strongly observable subgroup}

\subsubsection{} We consider the general linear group $GL_P$ of $P\in\cC_1$, see \cite[Example~7.1.5(1)]{ComAlg}. We denote by $H<GL_P$ the transporter subgroup of $\unit\simeq\soc P\subset P$, see \cite[7.4.4]{ComAlg}.

\begin{remark}
That $H<GL_P$ is not normal can be easily observed. For instance, consider $A\in\Alg\cC_1$ with some $a\in (\Omega A)^\times $ which is not in $A^0$, for example $x\in \Omega \Sym P$ in Example~\ref{SymP}. Consider $g\in GL_P(A)=\Aut_A(A\otimes P)$ corresponding to $1\otimes v_1\mapsto a\otimes v_1$ and $h\in H(A)$ corresponding to $1\otimes v_1\mapsto 1\otimes v_1+x\otimes v_2$, for some $x\in \Omega(A)$ with $(Dx)(Da)\not=0$, yields $ghg^{-1}$ not in $H(A)$.
\end{remark}

\begin{theorem}
\begin{enumerate}
\item The quotient scheme $GL_P/H$ exists as in Hypothesis~\ref{hypoQuo}. Moreover it is the affine $k$-scheme corresponding to $k[s]/s^2$.
\item  The immersion in $\Sch\cC_1$
$$GL_P/H\;\to\; \mP_P$$ from Lemma~\ref{LemImm} and Example~\ref{ExStab}(2), is given by $\bSpec$ of 
$$T\tto k[s]/s^2,\quad t\mapsto s.$$
\end{enumerate}
\end{theorem}
\begin{proof}
It is a straightforward exercise to verify that, for every $A\in\Alg\cC_1,$ the map
$$H(A)\backslash GL_P(A)\to (\Omega A)^\times/(A^0)^\times$$
which sends the equivalence class of an isomorphism of $A\otimes P$ given by $1\otimes v_1\mapsto a\otimes v_1+b\otimes v_2$ to $a (A^0)^\times $, is well-defined and a bijection. Part (1) thus follows from Corollary~\ref{Corses}(2) and $H(A)\backslash GL_P(A)\simeq GL_P(A)/H(A)$.


Part (2) follows from tedious verification, but also because the only non-zero morphisms $T\to k[s]/s^2$ send $t$ to a scalar multiple of $s$.
\end{proof}

\begin{remark}
One can verify easily that
$$H_{\red}\,=\, (GL_P)_{\red}\;\simeq\; \mG_m\times \mG_a.$$
\end{remark}

\subsection*{Acknowledgement}
The research was partially supported by ARC grant FT220100125. The author thanks Pavel Etingof, Victor Ostrik and Geordie Williamson for interesting discussions.

\end{document}